\documentclass[11pt,reqno,a4paper,oneside]{amsart}

\usepackage{latexsym}
\usepackage{amsmath}
\usepackage{amsthm}
\usepackage{amssymb}
\usepackage{amsfonts}
\usepackage{fancyhdr}
\usepackage{comment}
\usepackage{mathtools}

\usepackage{mathrsfs}                           
\DeclarePairedDelimiter\floor{\lfloor}{\rfloor}
\newcommand{\mc}{\mathcal}
\newcommand{\cal}{\mathcal}
\newcommand{\rr}{\mathbb{R}}
\newcommand{\R}{\mathbb{R}}
\newcommand{\nn}{\mathbb{N}}
\newcommand{\cc}{\mathbb{C}}
\newcommand{\C}{\mathbb{C}}

\newcommand{\zz}{\mathbb{Z}}

\newcommand{\eps}{\epsilon}

\newcommand{\pl}{\partial}
\newcommand{\x}{\times}

\newcommand{\til}{\widetilde}

\newcommand{\supp}{\textrm{supp}}
\newcommand{\cjd}{\rangle}
\newcommand{\cjg}{\langle}

\newcommand{\demi}{\frac{1}{2}}

\newcommand{\tra}{\textrm{Tr}}

\newcommand{\la}{\lambda}

\newcommand{\im}{\mathop{\hbox{\rm Im}}\nolimits}

\def\qed{\hfill$\square$\medskip}

\theoremstyle{definition}
\newtheorem{theorem}{Theorem}[section]
\newtheorem{proposition}{Proposition}[section]
\newtheorem{cor}{Corollary}[section]
\newtheorem{lemma}{Lemma}[section]

\newtheorem{remark}{Remark}[section]

\newcommand{\jb}[1]{\langle #1 \rangle}
\newcommand{\op}{\operatorname{Op}_h}
\newcommand{\ov}[1]{\overline{#1}}
\newcommand{\p}{\partial}

\newcommand{\hd}{\langle h D \rangle^{-1}}
\renewcommand\Re{\operatorname{Re}}
\renewcommand\Im{\operatorname{Im}}
\newcommand{\dbar}{\bar\partial}

\title[Inverse scattering for the magnetic Schr\"odinger
operator]{Inverse scattering for the magnetic Schr\"odinger
operator on surfaces with Euclidean ends}
\thanks{Leo Tzou is supported by Australian Research Council FT-130101346. Valter Pohjola is employed under Academy of Finland grant AKA-251469}

\begin{document}
\author{Valter Pohjola \and Leo Tzou}

\begin{abstract}
   We prove a fixed frequency inverse scattering result for the magnetic Schr\"odinger operator (or connection Laplacian) on surfaces with Euclidean ends. We show that, under suitable decaying conditions, the scattering matrix for the operator determines both the gauge class of the connection and the zeroth order potential.
\end{abstract}
\maketitle

\begin{section}{Introduction}
The purpose of this paper is to show that 
the scattering matrix $S_{X,V}(\lambda)$ of the magnetic Schr\"odinger
operator $(d+ iX)^*(d+iX) + V$ determines $V$ and the gauge class of the 1-form
$X$ on Riemann surfaces with Euclidean ends.

When $X= 0$ this was done by  in \cite{GST} and we
refer the reader to the article for all the relevant references and results in
this case.

In dimensions $n\geq3$ the problem of scattering by the magnetic Schr\"odinger operator was first considered in the simply connected setting by \cite{ER} in the smooth case and later by \cite{PSU} for less regular coefficients. As the setting is Euclidean, determining the gauge of $X$ is equivalent to determining its exterior derivative. Some cohomological aspects of this problem was considered in \cite{BW1, BW2} which described the Aharanov-Bohm effect using inverse scattering. These works still take place in the Euclidean setting and the topology is obtained by removing open balls from $\R^n$. Observe that when one assumes the coefficients are compactly supported, the inverse scattering problem is equivalent to the Calder\'on problem on the domain of support and this was done in \cite{NSU} for $n\geq 3$ and \cite{IUY2, GT2} for $n=2$.

In the present work we focus on the more geometric aspect of the problem where the ambient manifold is a general Riemann surface with Euclidean ends. We prove the following theorem
\begin{theorem}\label{thm1} 
Let $(M_0,g_0)$ be a non-compact Riemann surface with genus $\mathcal{G}$ and $N$ ends
isometric to $\rr^2\setminus\{|z|\leq 1\}$ 
with metric $|dz|^2$. Let $V_1,V_2 \in C^{1,\beta}(M_0)$ be two potentials with
$\beta>0$ and $X_1, X_2 \in H^{3+\eps_0}(M_0)$ for some $\eps_0>0$ be two 1-forms such that 
$S_{X_1, V_1}(\lambda) = S_{X_2, V_2}(\la)$ for some
 $\la\in\rr\setminus\{0\}$. Let $d(z,z_0)$ denote the distance  
between $z$ and a fixed point $z_0\in M_0$.\\ 
If  $N\geq \max(2\mathcal{G}+1,2)$, $V_j\in e^{-\gamma d(\cdot,z_0)}L^\infty(M_0)$, and
 $X_j\in e^{-\gamma d(\cdot,z_0)}H^{3+\eps_0}(M_0)$ for all $\gamma>0$, then there
 exists a unitary function $\Theta \in 1 + e^{-\gamma
 d(\cdot,z_0)}W^{1,\infty}(M_0)$ for all $\gamma >0$ such that $X_1 - X_2 =
 d\Theta /\Theta$ and $V_1=V_2$.\\
\end{theorem}
The approach we take in dealing with the non-trivial magnetic term $X_j$ is by
viewing $d+ iX_j$ as a unitary connection acting on the trivial line bundle
over $M_0$. Composition with the projection $\pi_{0,1} : T^*M_0 \to
T^*_{0,1}M_0$ yields a Cauchy-Riemann operator $(\bar\partial + iA_j):=
\pi_{0,1}(d+ iX)$ which, by \cite{kobayashi}, yields a unique complex structure
that can be trivialized by choosing a non-vanishing holomorphic section
$F_{A_j}$. 

The advantage of such trivializations is that it acts as a bridge between the two Cauchy-Riemann operators $(\bar\partial + iA_1)$ and $(\bar\partial + iA_2)$. Intuitively, this would effectively reduce the problem to the simpler case where $X_1 = X_2 = 0$ and one can then apply the techniques of \cite{GST}. Unfortunately, it is not always true that this conjugation between complex structures preserves the scattering information at infinity. We will see, in fact, that one can judiciously choose the trivializations to preserve this information precisely when the scattering matrices agree. It will become apparent that this is a global result about the trivialization rather than a local one of determining the asymptotic expansions of certain coefficients at infinity. In this sense the "boundary determination" performed here is quite different from those in \cite{WY, JSB, JSB1}.

This approach to treating Cauchy-Riemann operators was explored in \cite{T}, \cite{AGTU}, and \cite{GT2} for studying inverse boundary value problems. The setting of this article, however, is on a non-compact surface and therefore the previously used techniques for Calder\'on problems do not immediately apply.

The technique used here is the machinery of the b-calculus developed by Melrose in \cite{APS}. We will show that the condition $S_{X_1,V_2}(\lambda) = S_{X_2, V_2}(\lambda)$ induces an orthogonality relation between the difference of the trivializations and antiholomorphic 1-forms which belong to certain weighted $L^2$ spaces. Due to the results of b-calculus we can invoke Fredholm theory to show the existence of a holomorphic function which has the same expansion near infinity as the difference of the trivializations. 

In addition to this complication, the presence of a first-order term requires the construction of a different type of CGO and a different integral idenitity than those used in \cite{GST}. Particularly, in order to recover the gauge class of $X$ we will construct a class of CGOs which is compatible with the new boundary integral identity \[ \int_{M_0}\langle (|F_{A_1}|^{-2}- |F_{A_2}|^{-2} )\dbar \tilde v , \dbar \tilde w \rangle + \frac{1}{2} \langle (Q_1|F_{A_1}|^{2}-Q_2|F_{A_2}|^{2} ) \tilde v , \tilde w \rangle = 0,\] relating the size of the two trivializations $|F_{A_j}|$. The modulus of the trivializations turns out to carry all the information we need to recover the gauge class (see Proof of Proposition \ref{same gauge}).

The organization of this article is as follows. In Section \ref{sec_holomorphic} we prove
general facts about holomorphic functions and Fredholm properties of the Cauchy-Riemann operators on
weighted spaces. Section \ref{sec_ce} will be devoted to Carleman estimates on weighted spaces which can
produce higher regularity solvability results than those in \cite{GST}. In Section \ref{sec_scat} we develop the scattering theory for the
magnetic Schr\"odinger operator and construct the scattering matrix. We will
show in Section \ref{sec_bndry_id} that the scattering matrix determines the asymptotic
behaviour of the trivialization of the Cauchy-Riemann operator. This asymptotic
behaviour will be exploited in Section \ref{sec_conj} when we derive a new integral
identity which is more suitable for recovering the gauge class. In Section
\ref{sec_cgo} we will construct CGOs which we will then use in Sections \ref{sec_gauge} and
\ref{sec_zero} to recover the desired information.

\end{section}

\begin{section}{Holomorphic Morse functions on a surface with Euclidean ends} \label{sec_holomorphic}

\subsection{Riemann surfaces with Euclidean ends}
The contents of this section are similar to that in \cite{GST}. We include it here only for the convenience of the reader. 

Let $(M_0,g_0)$ be a non-compact connected smooth Riemannian surface with $N$ ends $E_1,\dots,E_N$ which are Euclidean, i.e. isometric to 
$\cc\setminus \{|z|\le 1\}$ with metric $|dz|^2$. By using a complex inversion $z\to 1/z$, each end is also isometric to a pointed disk
\[ E_i \simeq \{|z|\leq 1, z\not=0\} \textrm{ with metric }\frac{|dz|^2}{|z|^4}\]
thus conformal to the Euclidean metric on the pointed disk. The surface $M_0$ can then be compactified by adding the points corresponding  
to $z=0$ in each pointed disk corresponding to an end $E_i$, we obtain a closed Riemann surface $M$ with a natural complex structure induced by that of 
$M_0$, or equivalently a smooth conformal class on $M$ induced by that of $M_0$. Another way of thinking is to say that $M_0$ is the closed Riemann surface $M$ with $N$ points $e_1,\dots,e_N$ removed.   
The Riemann surface $M$ has  holomorphic charts $z_\beta:U_{\beta}\to \cc$
and we will denote by $z_1,\dots z_N$ the complex coordinates 
corresponding to the ends of $M_0$, or equivalently to the neighbourhoods of the points $e_i$.
 The Hodge star operator $\star$ acts on the cotangent bundle $T^*M$, its eigenvalues are 
$\pm i$ and the respective eigenspaces $T_{1,0}^*M:=\ker (\star+i{\rm Id})$ and $T_{0,1}^*M:=\ker(\star -i{\rm Id})$
are sub-bundles of the complexified cotangent bundle $\cc T^*M$ and the splitting $\cc T^*M=T^*_{1,0}M\oplus T_{0,1}^*M$ holds as complex vector spaces.
Since $\star$ is conformally invariant on $1$-forms on $M$, the complex structure depends only on the conformal class of $g$.
In holomorphic coordinates $z=x+iy$ in a chart $U_\beta$,
one has $\star(udx+vdy)=-vdx+udy$ and
\[T_{1,0}^*M|_{U_\beta}\simeq \cc dz ,\quad T_{0,1}^*M|_{U_\beta}\simeq \cc d\bar{z}  \]
where $dz=dx+idy$ and $d\bar{z}=dx-idy$. We define the natural projections induced by the splitting of $\cc T^*M$ 
\[\pi_{1,0}:\cc T^*M\to T_{1,0}^*M ,\quad \pi_{0,1}: \cc T^*M\to T_{0,1}^*M.\]
The exterior derivative $d$ defines the de Rham complex $0\to \Lambda^0\to\Lambda^1\to \Lambda^2\to 0$ where $\Lambda^k:=\Lambda^kT^*M$
denotes the real bundle of $k$-forms on $M$. Let us denote $\cc\Lambda^k$ the complexification of $\Lambda^k$, then
the $\pl$ and $\bar{\pl}$ operators can be defined as differential operators 
$\pl: \cc\Lambda^0\to T^*_{1,0}M$ and $\bar{\pl}:\cc\Lambda0\to T_{0,1}^*M$ by 
\[\pl f:= \pi_{1,0}df ,\quad \bar{\pl}f:=\pi_{0,1}df,\]
they satisfy $d=\pl+\bar{\pl}$ and are expressed in holomorphic coordinates by
\[\pl f=\pl_zf\, dz ,\quad \bar{\pl}f=\pl_{\bar{z}}f \, d\bar{z},\]  
with $\pl_z:=\demi(\pl_x-i\pl_y)$ and $\pl_{\bar{z}}:=\demi(\pl_x+i\pl_y)$.
Similarly, one can define the $\pl$ and $\bar{\pl}$ operators from $\cc \Lambda^1$ to $\cc \Lambda^2$ by setting 
\[\pl (\omega_{1,0}+\omega_{0,1}):= d\omega_{0,1}, \quad \bar{\pl}(\omega_{1,0}+\omega_{0,1}):=d\omega_{1,0}\]
if $\omega_{0,1}\in T_{0,1}^*M$ and $\omega_{1,0}\in T_{1,0}^*M$.
In coordinates this is simply
\[\pl(udz+vd\bar{z})=\pl v\wedge d\bar{z},\quad \bar{\pl}(udz+vd\bar{z})=\bar{\pl}u\wedge d{z}.\]
If $g$ is a metric on $M$ whose conformal class induces the complex structure $T_{1,0}^*M$, 
there is a natural operator, the Laplacian acting on functions and defined by 
\[\Delta f:= -2i\star \bar{\pl}\pl f =d^*d \]
where $d^*$ is the adjoint of $d$ through the metric $g$ and $\star$ is the Hodge star operator mapping 
$\Lambda^2$ to $\Lambda^0$ and induced by $g$ as well.
 
\subsection{Holomorphic functions}
We are going to construct Carleman weights given by holomorphic functions on $M_0$ which grow at most linearly or quadratically 
in the ends.  We will use the Riemann-Roch theorem, following ideas of \cite{GT}, however, the difference in the present case
is that we have very little freedom to construct these holomorphic functions, simply because there is just a finite dimensional space of such 
functions by Riemann-Roch.  
For the convenience of the reader, and to fix notations, we recall the usual Riemann-Roch index theorem 
(see Farkas-Kra \cite{FK} for more details). A divisor $D$ on $M$ is an element 
\[D=\big((p_1,n_1), \dots, (p_k,n_k)\big)\in (M\x\zz)^k, \textrm{ where }k\in\nn\]  
which will also be denoted $D=\prod_{i=1}^kp_i^{n_i}$ or $D=\prod_{p\in M}p^{\beta(p)}$ where $\beta(p)=0$
for all $p$ except $\beta(p_i)=n_i$. The inverse divisor of $D$ is defined to be 
$D^{-1}:=\prod_{p\in M}p^{-\beta(p)}$ and 
the degree of the divisor $D$ is defined by $\deg(D):=\sum_{i=1}^kn_i=\sum_{p\in M}\beta(p)$. 
A non-zero meromorphic function on $M$ is said to have divisor $D$ if $(f):=\prod_{p\in M}p^{{\rm ord}(p)}$ is equal to $D$,
where ${\rm ord}(p)$ denotes the order of $p$ as a pole or zero of $f$ (with positive sign convention for zeros). Notice that 
in this case we have $\deg(f)=0$.
For divisors $D'=\prod_{p\in M}p^{\beta'(p)}$ and $D=\prod_{p\in M}p^{\beta(p)}$, 
we say that $D'\geq D$ if $\beta'(p)\geq \beta(p)$ for all $p\in M$.
The same exact notions apply for meromorphic $1$-forms on $M$. Then we define for a divisor $D$
\[\begin{gathered}
r(D):=\dim (\{f \textrm{ meromorphic function on } M; (f)\geq D\}\cup\{0\}),\\
i(D):=\dim(\{u\textrm{ meromorphic 1 form on }M; (u)\geq D\}\cup\{0\}).
\end{gathered}\]
The Riemann-Roch theorem states the following identity: for any divisor $D$ on the closed Riemann surface $M$ of genus $\mathcal{G}$, 
\begin{equation}\label{riemannroch}
    r(D^{-1})=i(D)+\deg(D)-\mathcal{G}+1.
\end{equation}
Notice also that for any divisor $D$ with $\deg(D)>0$, one has $r(D)=0$ since $\deg(f)=0$ for all $f$ meromorphic. 
By \cite[Th. p70]{FK}, let $D$ be a divisor, then for any non-zero meromorphic 1-form $\omega$ on $M$, one has 
\begin{equation}\label{abelian}
i(D)=r(D(\omega)^{-1})
\end{equation}
which is thus independent of $\omega$. 
For instance, if $D=1$, we know that the only holomorphic function on $M$ is $1$ and 
one has $1=r(1)=r((\omega)^{-1})-\mathcal{G}+1$ and thus $r((\omega)^{-1})=\mathcal{G}$ if $\omega$ is a non-zero meromorphic $1$ form. Now 
if $D=(\omega)$, we obtain again from \eqref{riemannroch}
\[ \mathcal{G}=r((\omega)^{-1})=2-\mathcal{G}+\deg((\omega))\] 
which gives $\deg((\omega))=2(\mathcal{G}-1)$ for any non-zero meromorphic $1$-form $\omega$. In particular, if $D$ is a divisor such that
$\deg(D)>2(\mathcal{G}-1)$, then  we get 
$\deg(D(\omega)^{-1})=\deg(D)-2(\mathcal{G}-1)>0$ and thus $i(D)=r(D(\omega)^{-1})=0$, which implies by \eqref{riemannroch}
\begin{equation}\label{divisors}
    \deg(D)>2(\mathcal{G}-1)\Longrightarrow r(D^{-1})= \deg(D)-\mathcal{G}+1\geq \mathcal{G}.
\end{equation}
Now we deduce the 
\begin{lemma}\label{holofcts} 
    Let $e_1,\dots,e_{N}$ be distinct points on a closed Riemann surface $M$ with genus $\mathcal{G}$, and let $z_0$ be another 
    point of $M\setminus\{e_1,\dots,e_{N}\}$. If $N\geq \max(2\mathcal{G}+1,2)$, the following hold true:\\ 
(i) there exists a meromorphic function $f$ on $M$ with at most simple 
poles, all contained in $\{e_1,\dots,e_{N}\}$, such that $\pl f(z_0)\not=0$,\\
(ii) there exists a meromorphic function $f$
 on $M$ with at most simple poles, all contained in $\{e_1,\dots,e_{N}\}$, 
such that $z_0$ is a zero of order at least $2$ of $f$.
(iii) there exists a meromorphic function $f$ whose (non-removable) poles are all simple and form precisely the set $\{e_1,\dots,e_N\}$.

\end{lemma}
\noindent\textsl{Proof}. Let first $\mathcal{G} \geq 1$, so that $N \geq 2\mathcal{G}+1$. By the discussion before the Lemma, we know that there are at least $\mathcal{G}+2$ linearly independent (over $\cc$) 
meromorphic functions $f_0,\dots,f_{\mathcal{G}+1}$ on $M$ with at most simple poles, all contained in $\{e_1,\dots,e_{2\mathcal{G}+1}\}$. Without loss of generality,
one can set $f_0=1$ and by linear combinations we can assume that $f_1(z_0)=\dots=f_{\mathcal{G}+1}(z_0)=0$. Now for (ii) consider the divisor  
$D_j=e_1\dots e_{2\mathcal{G}+1}z_0^{-j}$ for $j=1,2$, with degree $\deg(D_j)=2\mathcal{G}+1-j$, then by the Riemann-Roch formula (more precisely \eqref{divisors})
\[ r(D_j^{-1})=\mathcal{G}+2-j.\]
Thus, since $r(D_1^{-1})>r(D_2^{-1})=\mathcal{G}$ and using the assumption that $\mathcal{G} \geq 1$, we deduce that there is a function in ${\rm span}(f_1,\dots,f_{\mathcal{G}+1})$ 
which has a zero of order $2$ at $z_0$ and a function which has a zero of order
exactly $1$ at $z_0$. To show (iii) observe that if $N \geq 2\mathcal{G}+1$ then
$r((e_1\dots e_N)^{-1})=N-\mathcal{G}+1$. Suppose none of the meromorphic functions with
divisor greater than or equal to $(e_1\dots e_N)^{-1}$ has a pole at $e_N$ then
one would have that $r((e_1\dots e_{N-1})^{-1})=N-\mathcal{G}+1$. But $\deg(e_1\dots
e_{N-1}) = N-1 \geq 2\mathcal{G}-1$ so $r((e_1\dots e_{N-1})^{-1})=N-\mathcal{G}$ by
\eqref{divisors}. This is a contradiction and therefore every point of
$\{e_1,\dots, e_N\}$ is a pole for some meromorphic function with divisor
greater than or equal to $(e_1\dots e_N)^{-1}$. Taking suitable linear
combination of these functions yields a meromorphic function with simple poles
precisely at the points $\{e_1,\dots,e_N\}$.

The same method clearly works if $\mathcal{G}=0$ by taking $N\geq2$. 
\qed
 

 
\subsection{Morse holomorphic functions with prescribed critical points} \label{morseholo}
We follow in this section the arguments used in \cite{GT} to construct holomorphic functions with non-degenerate critical points (i.e.~Morse holomorphic functions) on 
the surface $M_0$ with genus $\mathcal{G}$ and $N$ ends, such that these functions have at most linear growth 
in the ends if $N\geq \max(2\mathcal{G}+1,2)$. 
We let $\mc{H}$ be the complex  vector space spanned by the meromorphic functions on $M$ with divisors larger or equal to  
$e_1^{-1}\dots e_{N}^{-1}$ 
where $e_1,\dots e_{N}\in M$ are points corresponding to the ends of $M_0$ as explained in the previous section.  
Note that $\mc{H}$ is a complex vector space of complex dimension greater or equal to $N-\mathcal{G}+1$ 
for the $e_1^{-1}\dots e_{N}^{-1}$ divisor. 
We will also consider the real vector space $H$ spanned by the real parts and imaginary parts of functions in $\mc{H}$, this is a real vector space 
which  admits a  Lebesgue measure. We now prove the following 
\begin{lemma}\label{morsedense}
The set of functions $u\in H$ which are not Morse in $M_0$ has measure $0$ in $H$, in particular its complement is dense in $H$.
\end{lemma}
\noindent{\bf Proof}. We use an argument very similar to that used by Uhlenbeck \cite{Uh}.
We start by defining $m: M_0\times H\to T^*M_0$ by $(p,u) \mapsto (p,du(p))\in T_p^*M_0$. 
This is clearly a smooth map, linear in the second variable, moreover $m_u:=m(.,u)=(\cdot, du(\cdot))$ is  
smooth on $M_0$. The map $u$ is a Morse function if and only if 
$m_u$ is transverse to the zero section, denoted $T_0^*M_0$, of $T^*M_0$, i.e.~if 
{\Small\[\textrm{Image}(D_{p}m_u)+T_{m_u(p)}(T_0^*M_0)=T_{m_u(p)}(T^*M_0),\quad \forall p\in M_0 \textrm{ such that }m_u(p)=(p,0).\]}
This is equivalent to the fact that the Hessian of $u$ at critical points is 
non-degenerate (see for instance Lemma 2.8 of \cite{Uh}). 
We recall the following transversality result, the proof of which is contained in \cite[Th.2]{Uh} by replacing Sard-Smale theorem by the usual 
finite dimensional Sard theorem:
\begin{theorem}\label{transv}
Let $m : X\times H \to W$ be a $C^k$ map and $X, W$ be smooth manifolds and $H$  a finite dimensional vector space, 
if  $W'\subset W$ is a submanifold such that $k>\max(1,\dim X-\dim W+\dim W')$, then the transversality of the map $m$ to $W'$  implies that 
the complement of the set $\{u\in H; m_u \textrm{ is transverse to } W'\}$ in $H$ has Lebesgue measure $0$.
\end{theorem} 
We want to apply this result with $X:=M_0$, $W:=T^*M_0$ and $W':=T^*_0M_0$, and with the map $m$ as defined above. 
We have thus proved our Lemma if one can show that $m$ is transverse to $W'$. 
Let $(p,u)$ such that $m(p,u)=(p,0)\in W'$. Then identifying $T_{(p,0)}(T^*M_0)$ with $T_pM_0\oplus T^*_pM_0$, one has
\[Dm_{(p,u)}(z,v)=(z,dv(p)+{\rm Hess}_p(u)z)\]
where ${\rm Hess}_p(u)$ is the Hessian of $u$ at the point $p$, viewed as a linear map from $T_pM_0$ to $T^*_pM_0$ (note that this is different from the covariant Hessian defined by the Levi-Civita connection). 
To prove that $m$ is transverse to $W'$ we need to show that $(z,v)\to (z, dv(p)+{\rm Hess}_p(u)z)$ is onto from $T_pM_0\oplus H$ 
to $T_pM_0\oplus T^*_pM_0$, which is realized if the map $v\to dv(p)$ from $H$ to  $T_p^*M_0$ is onto.
But from Lemma \ref{holofcts}, we know that there exists a meromorphic function $f$ with real part $v={\rm Re}(f)\in H$ 
such that $v(p)=0$ and $dv(p)\not=0$ as an element of $T^*_pM_0$. We can then take $v_1:=v$ and $v_2:={\rm Im}(f)$, which are 
functions of $H$ such that $dv_1(p)$ and $dv_2(p)$ are
linearly independent in $T^*_pM_0$ by the Cauchy-Riemann equation $\bar{\pl} f=0$. 
This shows our claim and ends the proof by using Theorem \ref{transv}.\qed

In particular, by the Cauchy-Riemann equation, Lemma \ref{morsedense} implies
\begin{cor}\label{morse holomorphic dense}
The subset of functions in $\mc{H}$ which are Morse is dense.
\end{cor}
This shows, in conjunction with part (iii) of Lemma \ref{holofcts} that
\begin{lemma}
\label{poles at ends}
Let $e_1,\dots,e_{N}$ be distinct points on a closed Riemann surface $M$ with genus $\mathcal{G}$. If $N\geq \max(2\mathcal{G}+1,2)$, then there exists a morse meromorphic function $f$ whose (non-removable) poles are all simple and form precisely the set $\{e_1,\dots,e_N\}$. 
\end{lemma}

This discussion allows us to conclude that
\begin{proposition}
\label{criticalpoints}
There exists a dense set of points $p$ in $M_0$ such that there exists a Morse holomorphic function $f\in\mc{H}$ on $M_0$ whose (non-removable) poles are all simple and form precisely the set $\{e_1,\dots,e_N\}$ which has a critical point at $p$.
\end{proposition} 
\noindent\textsl{Proof}.  
Let $p$ be a point of $M_0$ and let $u$ be a holomorphic function with a zero of order at least $2$ at $p$, 
the existence is ensured by Lemma \ref{holofcts}. 
Let $B(p,\eta)$ be a any small ball of radius $\eta>0$ near $p$, then by 
Lemma \ref{morsedense}, for any $\eps>0$, we can approach $u$ by a holomorphic Morse function $u_\eps\in\mc{H}_\eps$ whose (non-removable) poles are all simple and form precisely the set $\{e_1,\dots,e_N\}$ and which is at distance less than $\eps$ of $u$ in a fixed norm on the finite dimensional space $\mc{H}$. 
Rouch\'e's theorem for $\pl_z u_\eps$ and $\pl_zu$ (which are viewed as functions locally near $p$) 
implies that $\pl_z u_{\eps}$ has at least one zero of order exactly $1$ in $B(p,\eta)$ if $\eps$ is chosen small enough. 
Thus there is a Morse function in $\mc{H}$ with a critical point  arbitrarily close to $p$.\qed

\begin{remark}
In the case where the surface $M$ has genus $0$ and $N$ ends,  we have an explicit formula for the function in Proposition \ref{criticalpoints}: indeed 
$M_0$ is conformal to $\cc\setminus \{e_1,\dots, e_{N-1}\}$ for some
$e_i\in\cc$ - i.e.~the Riemann sphere minus $N$ points - then the function $f(z)=(z-z_0)^2/(z-e_1)$ with $z_0\not\in\{e_1,\dots,e_{N-1}\}$ 
has $z_0$ for unique critical point in 
$\cc\setminus \{e_1,\dots, e_{N-1}\}$ and it is non-degenerate. 
\end{remark}

We end this section by the following Lemmas which will be used for the amplitude of the complex geometric optics solutions but not for the phase.
\begin{lemma}\label{amplitude}
For any $p_0, p_1,\dots p_n\in M_0$ some points of $M_0$ and $L\in\nn$, then there exists a function $a(z)$ holomorphic on $M_0$ 
which vanishes to order $L$ at all $p_j$ for $j=1,\dots,n$ and such that $a(p_0)\not=0$.
 Moreover $a(z)$ can be chosen to have  at most polynomial growth in the ends, i.e.
$|a(z)|\leq C|z|^{J}$ for some $J\in\nn$. The analogous statement can be made about holomorphic 1-forms.
\end{lemma}
\noindent\textsl{Proof}. It suffices to find on $M$ some meromorphic function with divisor greater or equal to  
$D:=e_1^{-J}\dots e_N^{-J}p_1^L\dots p_n^{L}$ but not greater or equal to $Dp_0$ 
and this is insured by Riemann-Roch theorem as long as  $JN-nL\geq 2\mathcal{G}$ since then 
$r(D)=-\mathcal{G}+1+JN-nL$ and $r(Dp_0)=-\mathcal{G}+JN-nL$.
\qed

\begin{lemma}
\label{control the zero}
Let $\{p_0, p_1,..,p_n\}\subset M_0$ be a set of $n+1$ disjoint points. Let $c_0,c_1,\dots, c_K\in\cc$, $L\in\nn$, 
and let $z$ be a complex coordinate near $p_0$ such that $p_0=\{z=0\}$. 
Then  there exists a holomorphic function $f$ on $M_0$ with zeros of order at least $L$ at each $p_j$,
such that $f(z)=c_0 + c_1z +...+ c_K z^K+ O(|z|^{K+1})$ in the coordinate $z$.  Moreover $f$ 
can be chosen so that there is $J\in\nn$ such that, in the ends, $|\pl_z^\ell f(z)|=O (|z|^{J})$ for all $\ell\in\nn_0$.
\end{lemma}
\noindent\textsl{Proof}. The proof goes along the same lines as in Lemma \ref{amplitude}. By induction on $K$ and linear combinations, 
it suffices to prove it for $c_0=\dots=c_{K-1}=0$. As in the proof of Lemma \ref{amplitude}, if $J$ is taken large enough, there exists 
a function with divisor greater or equal to $D:=e_1^{-J}\dots e_N^{-J}p_0^{K-1}p_1^L\dots p_n^{L}$ but not greater or equal 
to $Dp_0$. Then it suffices to multiply this function by $c_K$ times the inverse of 
the coefficient of $z^K$ in its Taylor expansion at $z=0$.
\qed

\subsection{Laplacian on weighted spaces} 
Let $x$ be a smooth positive function on $M_0$, which is equal to $|z|^{-1}$
for $|z|>r_0$ in the ends 
$E_i\simeq \{z\in\cc; |z|>1\}$, where $r_0$ is a large fixed number. 
We now show that the Laplacian $\Delta_{g_0}$ on a surface with Euclidean ends
has a right inverse on the weighted spaces 
$x^{-J}L^2(M_0)$ for $J\notin\nn$ positive.
\begin{lemma}\label{rightinv}
For any $J>-1$  which is not an integer, there exists a continuous operator $G$
mapping $x^{-J}L^2(M_0)$ to $x^{-J-2}L^2(M_0)$  
such that $\Delta_{g_0}G={\rm Id}$.
\end{lemma}
\noindent\textsl{Proof}. Let $g_b:=x^2g_0$ be a metric conformal to $g_0$. The
metric $g_b$ in the ends can be written 
$g_b=dx^2/x^2+d\theta_{S^1}^2$ by using radial coordinates $x=|z|^{-1},\theta=z/|z|\in S^1$,  
this is thus a b-metric in the sense of Melrose \cite{APS}, giving the surface
a geometry of surface with cylindrical ends. Let us define for $m\in\nn_0$
\begin{small}
\[
    H^m_b(M_0):=\{u\in L^2(M_0;{\rm dvol}_{g_b}); (x\pl_x)^j \pl_\theta^ku\in
    L^2(M_0;{\rm dvol}_{g_b}) \textrm{ for all }j+k\leq m\}.
\] 
\end{small}
The Laplacian has the form $\Delta_{g_b}=-(x\pl_x)^2+\Delta_{S^1}$ in the ends,
and 
the indicial roots of $\Delta_{g_b}$ in the sense of Section 5.2 of \cite{APS}
are given by the complex numbers $\la$ such that
$x^{-i\la}\Delta_{g_b}x^{i\la}$ is not invertible as an operator acting on the
circle $S_{\theta}^1$. Thus the indicial roots are the solutions of
$\la^2+k^2=0$ 
where $k^2$ runs over the eigenvalues of $\Delta_{S^1}$, that is, $k \in \zz$.
The roots are simple at $\pm i k \in i\zz\setminus \{0\}$ and $0$ is a 
double root. In Theorem 5.60 of \cite{APS}, Melrose proves that $\Delta_{g_b}$
is Fredholm on $x^{a}H^2_b(M_0)$ if and only if $-a$ is not the imaginary part
of some indicial root, that is here $a\not\in \zz$.  
For $J>0$, the kernel of $\Delta_{g_b}$ on the space $x^{J}H^2_b(M_0)$ is
clearly trivial by an energy estimate. Thus $\Delta_{g_b}: x^{-J} H^0_b(M_0)
\to x^{-J} H^{-2}_b(M_0)$ is surjective for $J > 0$ and $J \not\in \zz$, and
the same then holds for $\Delta_{g_b}: x^{-J} H^2_b(M_0) \to x^{-J}
H^{0}_b(M_0)$ by elliptic regularity.

Now we can use Proposition 5.64 of \cite{APS}, which asserts, for all positive $J \not\in \zz$, the existence of a  pseudodifferential operator $G_b$ mapping continuously $x^{-J}H^0_b(M_0)$ to $x^{-J}H^2_b(M_0)$ such  that $\Delta_{g_b}G_{b}={\rm Id}$.
Thus if we set $G=G_bx^{-2}$, we have $\Delta_{g_0}G={\rm Id}$ and $G$ maps continuously 
$x^{-J+1}L^2(M_0)$ to $x^{-J-1}L^2(M_0)$ (note that $L^2(M_0)=xH^0_b(M_0)$). 
\qed
\subsection{Cauchy-Riemann Operator on Weighted Space}
We begin this section with a discussion about the Fredholm properties of the operator $\bar\pl$ on non-compact manifolds by using the results of b-calculus. If $M_0$ is a surface with $N$ Euclidean ends, then one may take the $N$ point compactification to obtain a closed surface $M = M_0 \cup \{e_1,\dots, e_N\}$. In a holomorphic coordinate neighbourhood $E_j$ of $e_j$ the metric can be written in polar coordinates as $g_0 = \frac{dx^2}{x^4} + \frac{d\theta^2}{x^2}$. By extending $x$ to be a smooth positive function on $M_0$ one obtains a b-metric defined by $g_b := x^2 g_0$. In this setting $M_0$ is a bordered manifold with $x$ as its boundary defining function.

Let ${\mathcal V}_b$ denote the sections of $TM_0$ which are tangential to $\partial M_0$ at the boundary and $^bTM_0$ be the bundle so that ${\mathcal V}_b = C^\infty(M_0, ^bTM_0)$. Denoting its dual bundle by $^bT^*M_0$ one sees that near $ x=0$ the b-tangent and b-cotangent bundles are spanned by unitary (co)vectors $\{x\pl_x, \pl_\theta\}$ and $\{\frac{dx}{x}, d\theta\}$ respectively. The complexified b-cotangent bundle $\C ^bT^*M_0$ has a splitting into $^bT_{0,1}^*M_0 \oplus ^bT_{1,0}^*M_0$ given by the complex structure induced by $g_b$.

The Cauchy-Riemann operator is invariant within a conformal class of metric. However, for all $u\in C^\infty (M_0)$ it is convenient to express $\bar\pl u$ locally near $x=0$ as a section of $^bT^*_{0,1}M_0$:
\[\bar\pl u = (\pl_x u + \frac{i}{x} \pl_\theta  u)(dx + ix d\theta) = (x\pl_x u + i\pl_\theta u)(\frac{dx}{x} + id\theta).\]
Written in this way one sees that $\bar\pl \in {\rm Diff}_b^1(M_0; \C, ^bT^*_{0,1} M_0)$. It is also elliptic with indicial family 
\[I_x(\bar\pl, s) u = \frac{i}{2}(-s u + \pl_\theta u)  (\frac{dx}{x} +
id\theta)\]
which has simple roots whenever $s \in i\zz$ by taking $u(0,\theta) =
e^{-s\theta}$. We can therefore conclude by Theorem 5.60 \cite{APS} that
$\bar\partial: x^J H^m_b (M_0,\C) \to x^J H^{m-1}_b (M_0, ^bT^*_{0,1}M_0)$ is
Fredholm whenever $J \notin \zz$ where 
\begin{small}
\[
    x^JH^m_b(M_0):=\{u\in L^2(M_0;{\rm dvol}_{g_b}); (x\pl_x)^j \pl_\theta^ku\in
x^J L^2(M_0;{\rm dvol}_{g_b}) \forall j+k\leq m\}.
\]
\end{small}

We are now in a position to characterize the range of the $\bar\pl$ operator in these weighted Sobolev spaces. Indeed, if we denote by $^b\bar\pl^*$ the adjoint of $\bar\pl$ under the metric $g_b$ we see that for $J \notin \zz$ and $m\geq 1$,
{\Small \[R_{J,m}(\bar\pl) = \{\omega \in x^J H^{m-1}_b(M_0, ^bT^*_{0,1}M_0) ; \int_{M_0} \langle \omega,\eta\rangle_{g_b} {\rm dvol}_{g_b} = 0, \forall \eta \in N_{-J, -m+1}(^b\bar\pl^* ) \}.\]}
Here, for all $J,m\in\R$, $R_{J,m}(\bar\partial)$ and $N_{J, m}(^b\bar\pl^* )$ denote respectively the range and kernel of the operators $\bar\pl$ and $^b\bar\pl^*$ acting on their respective sections in $x^J H^m_b$. By elliptic regularity (Theorem 5.61 and (5.165) in \cite{APS}) this becomes
{\Small\begin{eqnarray}
\label{ortho to all holomorphic bform}
R_{J,m}(\bar\pl) = \{\omega \in x^J H^{m-1}_b(M_0, ^bT^*_{0,1}M_0) ; \int_{M_0} \langle \omega,\eta\rangle_{g_b} {\rm dvol}_{g_b} = 0, \forall \eta \in N_{-J, m}(^b\bar\pl^* ) \}
\end{eqnarray}}
We look at the relationship between $N_{J,m}(^b\bar\pl^*)$ and the null space of $\bar\pl^*$ acting on $T^*_{0,1}M_0$ when $J\in \R\backslash \zz$. If $\eta \in x^JH^m_b(M_0, ^bT^*_{0,1}M_0)$ then it is a weighted $H^m$ section of the bundle $^bT^*_{0,1}M_0$ which is a subspace of the dual space of the bundle whose smooth sections are the vector fields tangent to the boundary. Therefore, locally in the interior $\eta$ has coordinate expression $\eta = ud\bar z$ with $ u \in H^m$ and thus $\eta$ is a $H_{loc}^m$ section of $T^*_{0,1}M_0$. Near the boundary where $x = 0$, $\eta$ has the coordinate expression $\eta = u \frac{d\bar z}{\bar z} = u (\frac{dx}{x} + id\theta)$ where $u \in x^J H^m_b(M_0)$. Therefore, if $\eta \in N_{J,m}(^b\bar\pl^*)$ then $u$ is an antiholomorphic function satisfying $\int_{|x| <1} |x^{-J}u|^2 \frac{1}{x}dxd\theta <\infty$. Taking the Laurent series expression for $u$ we have that $u$ must have a zero of at least order $\lceil J\rceil$ at each end which implies that $\eta = u \frac{d\bar z}{\bar z}$ has a zero of order at least $\floor*{J}$. This means that
\begin{eqnarray}
\label{relate b-bundle to original bundle}
\eta \in N_{J,m}(^b\bar\pl^*) \Rightarrow \bar\pl^*\eta  =0 \ {\rm and }\ \eta \in x^J L^2(M_0, {\rm dvol}_{g_0}).
\end{eqnarray}
Furthermore, combining this discussion with standard argument about removability of singularities gives the the following Lemma and its Corollaries:
 \begin{lemma}
 \label{integrability implies zero}
 If $\eta \in N_{J,m}(^bT^*_{0,1}M_0)$ for $J>0$ then $\eta$ extends antiholomorphically to $M = M_0\cup \{e_1,\dots,e_N\}$ with zeroes of order at least $\floor*{J}$ at each of the ends $e_j$, $j = 1, \dots, N$. 
 \end{lemma}
 \begin{cor}
 \label{empty kernel}
Let $M_0$ be a surface with $N \geq 2{\mathcal G}+1$ ends. If $J\in\R\backslash \zz$ satisfies $J >1$ then $N_{J,m}(^b\bar\partial^*)$ is trivial.
 \end{cor}
 \begin{proof}
Lemma \ref{integrability implies zero} implies that $\eta$ can be extended antiholomorphically to a section of $T^*_{0,1}M$ by taking its value to be zero at $e_j$ for $j = 1,\dots, N$. If $N \geq 2{\mathcal G}+1$ this would force its degree $(\eta)$ to be greater than or equal to $2{\mathcal G}+1$ and thus forcing it to be the trivial section. 
 \end{proof}
 \begin{cor}
 \label{surjectivity of dbar}
 Let $M_0$ be a surface with $N \geq 2{\mathcal G}+1$ ends. If $J\in \R\backslash \zz$ satisfies $2>J >1$ then $R_{-J,m}(\bar\partial) = x^{-J} H_b^{m-1}(M_0; ^bT_{0,1}^*M_0)$. Furthermore, there exists a bounded operator \[\bar\partial^{-1}_J : x^{-J}H_b^{m-1}(M_0; ^bT_{0,1}^*M_0) \to x^{-J}H_b^{m}(M_0; \C)\] satisfying $\bar\partial \bar\partial^{-1}_J = Id$, $\bar\partial^{-1}_J \bar\partial = Id - \Pi$ where \[\Pi : x^{-J}L_b^{2}(M_0; \C) \to N_{-J,m}(\bar\partial) \subset x^{-J}H_b^{m}(M_0; \C)\] is the orthogonal projection on $x^{-J}L^2_b$ with respect to the inner product \[ \langle u, v\rangle_{J} := \int_{M_0} x^{2J} \bar u v {\rm dvol}_{g_b}.\]
 \end{cor}
 \begin{proof}
 Combine Corollary \ref{empty kernel} and \eqref{ortho to all holomorphic bform} we have that the operator \[\bar\partial: x^{-J}H^m_b(M_0, \C) \to x^{-J} H_b^{m-1}(M_0; ^bT_{0,1}^*M_0)\] is surjective. Theorem 5.60 \cite{APS} also states that this operator is Fredholm so there exists a generalized inverse \[\bar\partial^{-1}_J : x^{-J}H_b^{m-1}(M_0; ^bT_{0,1}^*M_0) \to x^{-J}H_b^{m}(M_0; \C)\]  
satisfying $\bar\partial \bar\partial^{-1}_J = Id$, $\bar\partial^{-1}_J \bar\partial = Id - \Pi$ where \[\Pi : x^{-J}H_b^{m}(M_0; \C) \to N_{-J,m}(\bar\partial) \subset x^{-J}H_b^{m}(M_0; \C)\] is the orthogonal projection described in the statement of the Corollary. \end{proof}
In the case when $\eta$ is compactly supported we can easily work out the expression for the kernel of $\bar\pl^{-1}_J$ by using the existing machinery. Indeed, if $K = \supp (\eta)$ is contained in the interior of $M_0$ then let $\{\chi_j\}_{j = 1}^l$ be a partition of unity by $C^\infty_0(M_0)$ functions for some open cover of $K$ by holomorphic coordinate neighbourhoods $\{U_j\}_{j=1}^l$ with $U_j \subset\subset M_0$. Let $\chi'_j$ be compactly supported smooth function in $U_j$ which is equal to $1$ in a neighbourhood of the support of $\chi_j$. Define \begin{eqnarray}\label{def of T}T \eta := \sum\limits_{j}\chi_j' \int_{U_j} \frac{\chi_j(z') f_j(z')}{z' -z} dz' \wedge d\bar z'\end{eqnarray} where $f_j d\bar z$ is the coordinate expression of $\eta$ in $U_j$. One immediately gets that 
\[\bar\pl T \eta = \eta + \big(\sum\limits_{j} \omega_j\int_{U_j} \kappa_j(z',z) f_j(z') dz' \wedge d\bar z'\big) \]
where $\kappa_j$ are smooth compactly supported functions in $U_j\times U_j$ and $\omega_j$ are smooth sections of $T^*_{0,1}M_0$ compactly supported in $U_j$. Hitting both sides with $\bar\pl_J^{-1}$ for $2>J>1$ we get from Corollary \ref{surjectivity of dbar}
\begin{eqnarray}
\label{expression for dbar inverse}
\bar\pl_J^{-1} \eta=(Id + \Pi)T \eta - \bar\pl_J^{-1} \sum\limits_{j} \omega_j\int_{U_j} \kappa_j(z',z) f_j(z') dz' \wedge d\bar z'
\end{eqnarray}
The expression \eqref{expression for dbar inverse} combined with the explicit formula for the $T$ operator given by \eqref{def of T} allows us to prove the following
\begin{lemma}
\label{estimate for conjugate dbar}
If $\psi$ is a Morse function on $M_0$ and $\eta$ is a compactly supported smooth section of $T^*_{0,1}M_0$ then for $2 > J >1$ we have 
\[\|x^J \bar\pl_J^{-1} (e^{2i\psi/h} \eta)\|_{L_b^2(M_0)} \leq C h^{\frac{1}{2} + \epsilon}.\]
Here the constant depends on $\eta$ and the size of its support.
\end{lemma}
\begin{proof}
Using the expression \eqref{expression for dbar inverse} and replacing $\eta$ by $e^{i2\psi/h}\eta$ we see that $\bar\partial_J^{-1}e^{-2i\psi/h}\eta$ can be bounded by two separate terms. The operator $\bar\pl^{-1}_J$ is bounded from $x^{-J} H_b^{m-1}(M_0; ^bT^*_{0,1}M_0)$ to $x^{-J} H^{m}_b(M_0)$ and the last term is the composition of this operator with a finite sum of integrals agains smooth compactly supported kernels. Therefore, the last term can be treated with stationary phase to show that its $x^{-J} L^2_b$ norm is of order $h$.  

The first term can be estimated by using the explicit expression of the kernel given in \eqref{def of T} and the fact that $(Id + \Pi)$ is bounded on $x^{-J} L^2_b(M_0)$. Since each of $\chi_j'$ are compactly supported and we are summing over finitely many terms in \eqref{def of T}, repeating the same argument as Lemma 2.2 of \cite{GT2} would yield that $\|x^JTe^{2i\psi/h}\eta\|_{L_b^2} \leq  C(\supp(\eta)) h^{\frac{1}{2} + \epsilon} \|\eta\|_{W^{1,p}}$
and the proof is complete by the boundedness of $(Id - \Pi)$. \end{proof}

 \begin{subsection}{Construction of Conjugation Factor}
 If $A$ is a section of $T^*_{0,1}M_0$ then the operator $\bar\partial + iA$ is
 a Cauchy-Riemann operator acting on the trivial complex line bundle over
 $M_0$. By \cite{kobayashi} there exists a unique holomorphic structure which
 is compatible with this Cauchy-Riemann operator and it is trivialized by a
 non-vanishing section of the bundle. It is useful to construct an explicit
 form of this trivialization. In particular, if $\alpha$ is a differentiable
 function satisfying $\bar\partial \alpha = A$, then one has $ \bar\partial +
 iA = e^{-i\alpha} \bar\partial e^{i\alpha}$. It is important in this article
 to understand the asymptotic of the trivialization near the ends. This
 motivates us to consider the following construction:
\begin{lemma}
\label{solve for alpha near infinity}
Let $\eta\in e^{-\gamma|z|} H^{3+\eps_0}(\C)$ for all $\gamma >0$. We have the following
expansion for the parametrix of the Cauchy-Riemann operator
\begin{eqnarray} \label{dbar expansion}
\bar\partial^{-1} \eta (z) := \frac{1}{2\pi} \int_{\C} \frac{\eta(\zeta)}{ z-
\zeta}d\zeta \wedge d\bar\zeta 
= \frac{c_{-1}}{z} +\dots+\frac{ c_{-k}}{ z^{k}}
+ R_k(z),
\end{eqnarray}
where $R_k(z) := z^{-k}\dbar^{-1}(z^k\eta)$, with 
$|\partial_z^j R_k(z)| \leq
C_{j,k} |z|^{-1} \|e^{\gamma|\zeta|}\eta(\zeta)\|_{H^{1+j+\eps_0}}$,
for all $j = 0, 1, 2$.
\end{lemma}
\begin{proof}
One sees easily that 
\[
    \int_{\C} \frac{\eta(\zeta)}{ z- \zeta}d\zeta \wedge
    d\bar\zeta = \frac{1}{z}\int \eta d\zeta\wedge d\bar\zeta +  \frac{1}{z}
    \int_{\C} \frac{\zeta \eta(\zeta)}{z- \zeta}d\zeta\wedge d\bar\zeta
\]
so using the fact that $\eta$ decays super-exponentially we can iterate this
relation get the expansion \eqref{dbar expansion}. 

To get the estimate on the remainder we observe again that since $\eta$ decays
super-exponentially it suffices to do this for $k= 0$. This can be
done for $j=0$, by first observing that $\eta \in e^{-\gamma|z|}L^\infty(\C)$,
by the Sobolev embedding Theorem, 
if $ \eta \in e^{-\gamma|\zeta|}H^{1+\eps_0}$ and then
spliting the integral and estimating as follows
\[
    \int_{|\zeta|\leq \frac{|z|}{2}} \frac{\eta(z-\zeta)}{\zeta} d\zeta\wedge d\bar\zeta +
    \int_{|\zeta|\geq \frac{|z|}{2}} \frac{\eta(z-\zeta)}{\zeta} d\zeta\wedge d\bar\zeta
    \leq 
    C (|z| e^{-\gamma |z|} + |z|^{-1} ).
\]
For $j=1$ the estimate can be done by observing that
$[\partial_z,\bar\partial^{-1}]\eta$ is an entire function and furthermore
$\partial_z\bar\partial^{-1} \eta\in L^2$ by Calder\'on-Zygmund while
$\bar\partial^{-1} \partial_z \eta = O(|z|^{-1})$ by the $j= 0$ case. We
conclude then that $[\partial_z, \bar\partial^{-1}]\eta$ is an entire function
which is in $\langle |z|\rangle^\epsilon L^2(\C)$ for all $\epsilon>0$. This
forces $[\partial_z, \bar\partial^{-1}]\eta = 0$ which means
$\partial_z\bar\partial^{-1}\eta = \bar\partial^{-1}\partial_z \eta =
O(|z|^{-1})$ by the fact that $\eta\in  e^{-\gamma|z|} H^l(\C)$ and using the
$j= 0$ estimate. This argument can be made as well for for $j=2$ (in fact as many times as the differentiability of $\eta$ allows) and the proof is complete.
\end{proof}

\begin{lemma}
\label{solve for compact part of alpha}
Let $\eta\in H^1(M_0; T^*_{0,1} M_0)$ be a compactly supported 1-form. Then there
exists solutions to the equation $\bar\partial\alpha = \eta$ which has
uniformly convergent power series expansion 
\[
    \alpha = c_1 z + c_0 + \sum\limits_{j = 1}^\infty c_j z^{-j}
\] for $|z|$ large. 
\end{lemma}
\begin{proof}
Denoting by $x := |z|^{-1}$ it is clear that since $\eta$ is compactly
supported it belongs to $x^{-J} H_b^{1}(M_0; ^bT_{0,1}^*M_0)$ for $2>J>1$ and
by Corollary \ref{surjectivity of dbar} we see that there exists a unique
$\alpha \in x^{-J} H_b^{2}(M_0;\C)$ solving $\bar\partial \alpha = \eta$. As
$\eta$ is compactly supported, $\alpha$ is actually holomorphic for large $|z|$
in the ends $E_j$ and has a Laurent series expansion $\alpha = \sum\limits_{j =
-\infty}^\infty c_j z^j$ which converges uniformly in the annulus $R_1<|z|<R_2$
for $0<R_1< R_2$ large enough. By the fact that $\alpha \in  x^{-J}
H_b^{2}(M_0;\C)$ with $2>J >1$ this forces $c_j = 0$ for $j \geq 2$. 
\end{proof}
Combining Lemmas \ref{solve for alpha near infinity} and \ref{solve for compact
part of alpha} we obtain the following 
\begin{proposition}
\label{constructing alpha}
Let $\eta\in e^{-\gamma/x} H^l(M_0; T^*_{0,1}M_0)$ for all $\gamma >0$. There
exists a function $\alpha \in x^{-J} H^{l+1}(M_0;\C)$, $2>J>1$ solving $\bar\partial
\alpha = \eta$ which for any $k$ has expansion \[|\partial_z^j  \big(\alpha  -
(c_{1} z + c_0 + \dots + c_{-k} z^{-k})\big)|  \leq C_{j,k}|z|^{-(k+1)}\]
near the ends when $|z| \to \infty$ for $j = 0, 1, 2$.
\end{proposition}
\end{subsection}
\end{section}

\begin{section}{Carleman Estimates and Solvability} \label{sec_ce}

In this section, we prove a Carleman estimate using harmonic weights with
non-degenerate critical points. 
Our starting point is the estimates in \cite{GST}, which are used to obtain an 
estimate for negative order Sobolev spaces. 
Duality will then allows us to prove a $H^1_{scl}$ solvability result for the magnetic operator, that will
later needed in constructing complex geometric optics solutions.
We remind the reader that we use $\Delta_g$ to denote the positive Laplacian.

\medskip
\noindent
We first consider a Morse holomorphic function $\Phi\in \mc{H}$ obtained from
Proposition \ref{criticalpoints}
with the condition that $\Phi$ has linear growth in the ends. We will write
\begin{align} \label{eq_Phi}
    \Phi := \varphi + i \psi, \quad \text{ where } \varphi := \Re(\Phi),\, \psi:= \Im(\Phi).
\end{align}
The Carleman weight will consist of the harmonic function $\varphi = \mathrm{Re}(\Phi)$.
We let  $x$ be a positive smooth function  on $M_0$ such that  $x=|z|^{-1}$  in 
the complex charts $\{z\in\cc; |z|>1\}\simeq E_j$ covering the end $E_j$. We
will assume without loss of generality that $\Phi$ (and therefore $\varphi$)
does not have critical points in $\overline{E_j}$.\\ 

We modify our weight using a function $\varphi_0$.
Let $\delta\in(0,1)$ be small and let us take $\varphi_0\in
x^{-\beta}L^2(M_0)$ a solution of $\Delta_{g_0}\varphi_0=x^{2-\delta}$, a
solution exists 
by Proposition \ref{rightinv} if $\beta>1+\delta$. 
Actually, by using Proposition 5.61 of \cite{APS}, if we choose $\beta<2$,
then it is easy to see that 
$\varphi_0$ is smooth on $M_0$ and has polyhomogeneous expansion as $|z|\to
\infty$, with leading asymptotic in the end $E_i$ given by
$\varphi_0= -x^{-\delta}/\delta^2+
c_i\log(x)+ d_i+O(x)$ for some $c_i,d_i$ which are smooth functions in $S^1$.


We will modify our weight function one step further to allow more generality.
We assume that $\alpha$ is as in Proposition \ref{constructing alpha}, with
$\dbar \alpha = \eta \in e^{-\gamma/x}H^{3+\eps_0}(M_0)$. In particular 
that $\alpha$ has a leading asymptotics in
the end $E_j$ given by \[|\partial_z^j  \big(\alpha  - (c_{1} z + c_0 + \dots +
c_{-k} z^{-k})\big)|  \leq C_{j,k}|z|^{-(k+1)}\]
near the ends when $|z| \to \infty$ for $j = 0,1,2$.
For $\eps>0$ small, we define the convexified weight 
\begin{align} \label{eq_phieps}
    \varphi_\eps:=\varphi + h{\rm Re}(i\alpha)-\frac{h}{\eps}\varphi_0.
\end{align}
It follows that $\varphi_\eps$ has an expansion at infinity of the form 
\[
    \varphi_\eps(z)= \gamma.z+ \frac{h}{\eps} \frac{r^{\delta}}{\delta^2} +c_1\log(r)+ c_2+c_3 r^{-1}+O(r^{-2}),
\]
where $r=|z|,\omega=z/r, \gamma=(\gamma_1 + h\gamma_1',\gamma_2 +
h\gamma_2')\in \rr^2$, $z = (z_1, z_2) \in \R^2$, and $c_i$ are some smooth functions on $S^1$ depending
on $h$. Moreover we have that
\begin{equation}
\begin{aligned} \label{eq_phiasym}
    d\varphi_\eps &= \gamma_1dz_1+\gamma_2dz_2+O(r^{-1+\delta}), \; \\
 \pl_{z}^\kappa\pl_{\bar{z}}^\mu \varphi_\eps(z)&=O(r^{-2+\delta}) \,\,\textrm{ for all }  \kappa+\mu\geq 2.
\end{aligned} 
\end{equation}
%
The following estimate was proved in \cite{T} with $\gamma_1' = \gamma_2' = 0$
but as they are lower order terms in the phase and the domain one considers is
compact, the same proof holds in the slightly more general case of
$\varphi_\eps$. See Proposition 3.1 in \cite{T} for details.

\begin{proposition}
\label{cmpct_ce}
Let  $K\subset M_0$ be compact and the $\varphi_\eps$ the previously defined  weight.
Then for $u\in C_0^\infty(K)$, we have 
\begin{align*}\label{carlemaninK}
    \frac{Ch}{\eps}\Big(\sqrt{h} \| u\|_{L^2} 
    + \| d\varphi_{\eps}u\|_{L^2} + \|hdu\|_{H^{-1}_{scl}} \Big) 
    \leq \|e^{\varphi_{\eps}/h} h^2 \Delta_g e^{-\varphi_{\eps}/h} u\|_{H^{-1}_{scl}},
\end{align*} 
where $C$ depends on $K$ but not on $h$ and $\eps$.
\end{proposition}

\medskip
\noindent
We will use semiclassical pseudodifferential calculus in the following proofs.
A function $a \in C^{\infty}(\R^2 \times \R^2)$  
is in the semiclassical symbol class $\mathcal{S}^k(\jb{\xi}^m)$, if 
\begin{equation} \label{eq_symclass}
| \partial_x^{\alpha} \partial_{\xi}^{\beta} a(y,\xi;h)| 
\leq C_{\alpha \beta} h^{-k} \jb{\xi}^m,
\end{equation}
where $C_{\alpha \beta}$ is independent of the parameter $h$, see \cite{S} and \cite{Zw}.
We will use the abbreviation $\mathcal{S}^m := \mathcal{S}^0(\jb{\xi}^m) $.

\medskip
We will need to prove an analogue of Proposition \ref{cmpct_ce} for the ends. Combining these will give us
a global estimate in the semiclassical $H^{-1}$-norm. 
We begin by proving the following weighted $L^2$-estimate, which is essentially the same as
Proposition 3.1 
in \cite{GST}, apart from the more general weight function used here. We give the proof
here as a convenience to the reader.

\begin{proposition}\label{carlemaninend}
 Let $\delta\in(0,1)$, and $\varphi_\eps$ as above, then there exists $C>0$
 such that for all $\eps\gg h>0$ small enough, and all $u\in C_0^\infty(E_j)$
\[ 
    h^2||e^{\varphi_{\eps}/h}(\Delta-\la^2)e^{-\varphi_{\eps}/h}u||^2_{L^2}\geq
    \frac{C}{\eps} (||x^{1-\frac{\delta}{2}}u||^2_{L^2}+ 
    h^2||x^{1-\frac{\delta}{2}}du||^2_{L^2}).
\]
\end{proposition}
\noindent\textsl{Proof}.
The metric $g_0$ can be extended to $\rr^2$ to be the Euclidean metric and we
shall denote by $\Delta$ the flat positive Laplacian on $\rr^2$.  Let us write
$P:=\Delta_{g_0}-\la^2$, then the operator $P_h:=
h^2e^{\varphi_\eps/h}Pe^{-\varphi_\eps/h}$ is given by 
\[
    P_h=h^2\Delta-|d\varphi_\eps|^2+2h\nabla\varphi_\eps.\nabla- h\Delta\varphi_\eps-h^2\la^2, 
\]
is a semiclassical operator with a  semiclassical full
Weyl symbol 
\[
    \sigma_w(P_h):=|\xi|^2-|d\varphi_\eps|^2-h^2\la^2+2i\cjg d\varphi_\eps,\xi\cjd =a+ib \in \mathcal{S}^2.
\]
We can define
$A:=(P_h+P_h^*)/2=h^2\Delta-|d\varphi_\eps|^2-h^2\la^2$ and
$B:=(P_h-P_h^*)/2i=-2ih\nabla\varphi_\eps.\nabla+ih\Delta\varphi_\eps$ which
have respective semiclassical full symbols $a$ and $b$, i.e.  $A={\rm Op}_w(a)$
and $B={\rm Op}_w(b)$ for the Weyl quantization. Notice that $A,B$ are
symmetric operators, thus for all $u\in C_0^\infty(E_i)$
\begin{equation}\label{A+iB} 
    ||(A+iB)u||^2=\cjg (A^2+B^2+i[A,B])u,u\cjd .
\end{equation} 
It is easy to check that the operator $ih^{-1}[A,B]$ is a
semiclassical differential operator in $\mathcal{S}^2$ with full semiclassical symbol  
\begin{equation}\label{symbol[A,B]}
    \{a,b\}(\xi)=  4(D^2\varphi_\eps(d\varphi_\eps,d\varphi_\eps)+D^2\varphi_\eps(\xi,\xi))
\end{equation}

Let us now decompose the Hessian of $\varphi_\eps$ in the basis
$(d\varphi_\eps, \theta)$ where $\theta$ is a covector orthogonal to
$d\varphi_\eps$ and of norm $|d\varphi_\eps|$. This yields coordinates
$\xi=\xi_0d\varphi_\eps+\xi_1\theta$ and there exist smooth functions $M,N,K$
so that
\[
    D^2\varphi_\eps(\xi,\xi)=|d\varphi_\eps|^2(M\xi_0^2+N\xi_1^2+2K\xi_0\xi_1).
\]
The asymptotics in \eqref{eq_phiasym} imply that $M,N,K\in r^{-2+\delta}L^\infty(E_i)$.  Then one can write
\[
\begin{split} \{a,b\}= &
4|d\varphi_\eps|^2(M+M\xi_0^2+N\xi_1^2+2K\xi_0\xi_1)\\ =&4(N(a+h^2\la^2)
+((M-N)\xi_0+2K\xi_1)b/2+(N+M)|d\varphi_\eps|^2) \end{split}
\] 
and since
\[M+N=\tra(D^2\varphi_\eps)=-\Delta\varphi_\eps= h \Delta {\rm Re}(i\alpha) +
h\Delta\varphi_0/\eps =  \frac{h}{\eps}x^{2-\delta}  + h {\rm Re}(i \Delta\alpha)\]
with $\Delta\alpha =\bar\partial^* \eta \in e^{-\gamma|z|} H^{2+\eps_0}(\R^2) \subset e^{-\gamma|z|} W^{1,\infty}(\R^2)$ we obtain 
\begin{equation}\label{{a,b}}
\begin{gathered}
    \{a,b\}=4|d\varphi_\eps|^2(c(z)(a+h^2\la^2) +\ell(z,\xi)b+\frac{h}{\eps}
    r^{-2+\delta} + h \Delta {\rm Re}(i\alpha)),\\
    c(z)=\frac{N}{|d\varphi_\eps|^2}, \,\,\,
    \ell(z,\xi)=\frac{(M-N)\xi_0+2K\xi_1}{2|d\varphi_\eps|^2}.  
\end{gathered}
\end{equation} 
Now, we take a smooth extension of $|d\varphi_\eps|^2,
a(z,\xi),\ell(z,\xi), \alpha(z)$ and $r$ to $z\in\rr^2$, this can done for
instance by  extending $r$ as a smooth positive function on $\rr^2$ and  then
extending $d\varphi$ and $d\varphi_0$ to smooth non vanishing $1$-forms on
$\rr^2$ (not necessarily exact) so that $|d\varphi_\eps|^2$ is smooth positive
(for small $h$) and polynomial in $h$ and $a,\ell$ are of the same form as in
$\{|z|>1\}$.  Let us define the symbol and quantized differential operator on
$\rr^2$ \[ e:=4|d\varphi_\eps|^2(c(z)(a+h^2\la^2) +\ell(z,\xi)b), \quad E:={\rm
Op}_w(e)\] and write

\begin{equation}\label{simpl}
\begin{gathered}
ih^{-1}r^{1-\frac{\delta}{2}}[A,B]r^{1-\frac{\delta}{2}}= hF +
r^{1-\frac{\delta}{2}}Er^{1-\frac{\delta}{2}}-\frac{h}{\eps}(A^2+B^2),\\ 
\textrm{ with }F:= h^{-1}r^{1-\frac{\delta}{2}}
(ih^{-1}[A,B] -E)r^{1-\frac{\delta}{2}}+\frac{1}{\eps}(A^2+B^2).
\end{gathered}
\end{equation}
We deduce from \eqref{symbol[A,B]} and \eqref{{a,b}} the following
\begin{lemma}\label{propF} 
The operator $F$  is a semiclassical differential
operator in the class $\mathcal{S}^4$ with semiclassical principal
symbol 
\[
    \sigma_w (F)(\xi)= 4|d\varphi|^2(\frac{1}{\eps} +  r^{2-\delta}
    {\rm Re}(i\Delta\alpha))+\frac{1}{\eps}(|\xi|^2-|d\varphi|^2)^2+\frac{4}{\eps}(\cjg\xi,d\varphi\cjd)^2.
\] 
\end{lemma} 
By the semiclassical G{\aa}rding estimate, we obtain the 
\begin{cor}\label{garding} 
The operator $F$ of Lemma \ref{propF} is such that there is a constant $C$ so that \[ \cjg
Fu,u\cjd \geq \frac{C}{\eps} (||u||^2_{L^2} +h^2||du||^2_{L^2}).\]
\end{cor} \textsl{Proof}. It suffices to use that when $\eps>0$ is chosen to be
small enough, $\sigma_w(F)(\xi)\geq \frac{C'}{\eps}(1+|\xi|^4)$ for some $C'>0$
and use the semiclassical G{\aa}rding estimate. The symbol estimate comes from
the fact that $|d\varphi|$ is bounded away from $0$ and $\Delta {\rm
Re}(i\alpha)$ decays superexponentially.  \qed

So by writing $\cjg i[A,B]u,u\cjd=\cjg
ir^{1-\frac{\delta}{2}}[A,B]r^{1-\frac{\delta}{2}}
r^{-1+\frac{\delta}{2}}u,r^{-1+\frac{\delta}{2}}u\cjd$ in \eqref{A+iB} and
using \eqref{simpl} and Corollary \ref{garding}, we obtain that there exists
$C>0$ such that for all $u\in C_0^\infty(E_i)$ 
\begin{align} \label{phu}
        ||P_hu||_{L^2}^2
        &\geq  
        \cjg (A^2+B^2)u,u\cjd 
        + \frac{Ch^2}{\eps}(||r^{-1+\frac{\delta}{2}}u||^2_{L^2}
        + h^2||r^{-1+\frac{\delta}{2}}du||^2_{L^2}) \\
        &\quad + h\cjg Eu,u\cjd 
        - \frac{h^2}{\eps}(||A (r^{-1+\frac{\delta}{2}} u)||_{L^2}^2 +||B
        (r^{-1+\frac{\delta}{2}} u)||_{L^2}^2). \nonumber 
\end{align} 
We observe
that  $h^{-1}[A,r^{-1+\frac{\delta}{2}}]r^{1+\frac{\delta}{2}}\in \mathcal{S}^1$ 
and  $h^{-1}[B,r^{-1+\frac{\delta}{2}}]r^{1+\frac{\delta}{2}}\in h\mathcal{S}^0$, and thus 
\begin{small}
\begin{align*}
    ||A (r^{-1+\frac{\delta}{2}} u)||_{L^2}^2 +||B (r^{-1+\frac{\delta}{2}} u)||_{L^2}^2)
    & \leq
    C'(||Au||_{L^2}^2+||Bu||_{L^2}^2 \\
    &+ h^2||r^{-1+\frac{\delta}{2}}u||^2_{L^2}+h^4||r^{-1+\frac{\delta}{2}}du||^2_{L^2})
\end{align*}
\end{small}
for some $C'>0$. Taking $h$ small, this implies with \eqref{phu} that there
exists a new constant $C>0$ such that
\begin{align}\label{phu2}
    ||P_hu||_{L^2}^2 
    &\geq \frac{Ch^2}{\eps}(||r^{-1+\frac{\delta}{2}}u||^2_{L^2}+h^2||r^{-1+\frac{\delta}{2}}du||^2_{L^2})  \\ 
    &\quad+ \frac{1}{2}\cjg (A^2+B^2)u,u\cjd
    + h\cjg Eu,u\cjd. \nonumber
\end{align}
It remains to deal with $h\cjg Eu,u\cjd$: we first write
$E=4|d\varphi_\eps|^2(c(z)(A+h^2\la^2)+ {\rm
Op}_w(\ell)B)+hr^{-1+\frac{\delta}{2}}Sr^{-1+\frac{\delta}{2}}$ where $S$ is a
semiclassical differential operator in the class $\mathcal{S}^1$ by the
decay estimates on $c(z),\ell(z,\xi)$ as $z\to\infty$, then by Cauchy-Schwartz
(and with $L:={\rm Op}_w(\ell)$) 
\begin{align*}
        |\cjg hEu,u\cjd|
        &\leq 
        Ch(||Au||_{L^2}+h^2||r^{-1+\frac{\delta}{2}}u||_{L^2} + h||Sr^{-1+\frac{\delta}{2}}u||_{L^2}) 
         ||r^{-1+\frac{\delta}{2}}u||_{L^2} \\
        &\quad + Ch||Bu||_{L^2}||Lu||_{L^2}\\ \leq & \frac{1}{4}||Au||^2_{L^2}
        + h^2||Sr^{-1+\frac{\delta}{2}}u||^2_{L^2}+Ch^2||r^{-1+\frac{\delta}{2}}u||_{L^2}^2 
        +\frac{1}{4} ||Bu||^2_{L^2}\\
        &\quad + Ch^2||Lu||^2_{L^2} 
\end{align*}
where $C$ is a constant
independent of $h,\eps$ but may change from line to line.  Now we observe that
$Lr^{1-\frac{\delta}{2}}$ and $S$ are in $\mathcal{S}^1$ and thus
\[
    ||Sr^{-1+\frac{\delta}{2}}u||^2_{L^2}+ ||Lu||^2_{L^2}
    \leq C
    (||r^{-1+\frac{\delta}{2}}u||^2_{L^2}+ h^2||r^{-1+\frac{\delta}{2}}du||^2_{L^2}),
\] 
which by \eqref{phu2} implies that
there exists $C>0$ such that for all $\eps\gg h>0$ with $\eps$ small enough
\[
    ||P_hu||_{L^2}^2\geq  \frac{Ch^2}{\eps}(||r^{-1+\frac{\delta}{2}}u||^2_{L^2}
    +h^2||r^{-1+\frac{\delta}{2}}du||^2_{L^2}) 
\]
for all $u\in C_0^\infty(E_i)$ . The proof is complete.\qed\\


\medskip
In the following proofs
we need some additional facts about the semiclassical calculus.
Firstly recall that 
a symbol $a \in \mathcal{S}^m$  corresponds in the so called classical quantization
to an operator $\operatorname{Op}_h(a) = a(y,hD)$ defined
by
\begin{equation*}
\operatorname{Op}_h(a)  f(y) = (2\pi)^{-2} \int_{\R^2} e^{iy\cdot\xi} a(y,h\xi;h) \hat{f}(\xi) \,d\xi.
\end{equation*}
We use $\sigma(A)$ to denote the symbol
corresponding to to a semiclassical operator $A$.

Moreover we need a formula for the commutator of two semiclassical operators with symbols
$a \in \mathcal{S}^m$ and $b \in \mathcal{S}^{m'}$. 
We have that $\sigma([\op(a), \op(b)]) \in \mathcal{S}^{m+m'}$ and 
\begin{equation} \label{eq_comexp}
    \sigma([\op(a), \op(b)]) 
    = \frac{h}{i} (\nabla_{\xi} a \cdot \nabla_y b - \nabla_y a \cdot \nabla_{\xi} b) + h^{2} \mathcal{S}^{m+m'},
\end{equation}
See \cite{S}.
We shall moreover utilize the following Proposition from \cite{S}. 
It is convenient to  formulate this using the weighted semiclassical spaces
$H^s_{\delta,scl}$, defined by the norm $\|f\|_{H^s_{\delta,scl}} := \| \jb{hD}^s
\jb{y}^{\delta} f \|_{L^2}$.

\begin{proposition} \label{weightedPSDO}
    Let $a \in \mathcal{S}^0$ and $ \delta_0 \geq 0$.
Then $\op(a)$ is bounded $H^s_{\delta,scl}(\R^n) \to H^s_{\delta,scl}(\R^n)$
for any $s, \delta \in \R$, and there is a constant $C$ with
$\| \op(a) \|_{ H^s_{\delta,scl} \to H^s_{\delta,scl}} \leq C$ whenever 
$|s|\leq s_0$, $|\delta| \leq \delta_0$, and $0 < h \leq h_0$.
\end{proposition}

We now prove a weighted version of Proposition \ref{cmpct_ce} that holds in the ends. This
is done by shifting the estimate of Proposition  \ref{carlemaninend}.

\begin{proposition}
\label{shifted_ce}
Let $\delta\in(0,1)$, and $\varphi_\eps$ as above, then there exists $C>0$ such that for all
$\eps\gg h>0$ small enough, and all $u\in C_0^\infty(E_j)$,
\begin{align*}
    \frac{C}{\eps} \|x^{1-\frac{\delta}{2}}u\|_{L^2} 
    \leq h\|e^{\varphi_{\eps}/h}(\Delta-\la^2)e^{-\varphi_{\eps}/h}u\|_{H^{-1}_{scl}}.
\end{align*}
\end{proposition}
\begin{proof}
We will employ the same notations as in the proof of Proposition  \ref{carlemaninend}.
In particular $r= x^{-1}$ and $P_h := e^{\varphi_{\eps}/h}h^2(\Delta-\la^2)e^{-\varphi_{\eps}/h}$. 

Let $\chi \in C^\infty(\R^2)$, be such that $\chi(y) = 1 $, in $\R^2 \setminus \mathbb{D}_1$ and 
$\chi(y) = 0$, near $\overline{\mathbb{D}}_{1/2}$.
Now consider the function $\chi\hd u$, where $u \in C^\infty_0( \R^2 \setminus \mathbb{D}_1)$. 
It is straight forward to see, using a density argument that Proposition \ref{carlemaninend} applies to 
functions in the Schwartz class, so that we may apply it to $\chi\hd u$
and get that
\begin{align}  \label{eq_p12}
    \|r^{\frac{\delta}{2}-1}\chi\hd u\|_{L^2}
    + &h\|r^{\frac{\delta}{2}-1}d(\chi\hd u)\|_{L^2} \\
    &\leq C \eps 
    \| P_h(\chi\hd u)\|_{L^2}.  \nonumber
\end{align}
Let $\theta =\delta/2-1$, so that $r^{\frac{\delta}{2}-1} = r^\theta$. 
We want to estimate the left hand side from below, by $\|r^{\frac{\delta}{2}-1} u\|_{L^2}$.
This is equivalent to estimating it from below by $\| u\|_{L^2_\theta}$.
We start by writing
\begin{align*} 
    \| r^\theta \chi\hd u\|_{L^2}
    + &h\|r^\theta d(\chi\hd u)\|_{L^2}  \nonumber\\
    &\geq
    \| \chi \hd u\|_{H^1_{ \theta, scl}}  - h \|[d,r^\theta](\chi\hd u)\|_{L^2}, 
\end{align*}
We can absorb the commutator term by the first term on the left hand side, when $h$
is small, since $[d,r^\theta]=\theta r^{\theta-1}$.
It is hence enough to estimate the term
containing the weighted Sobolev norm from below. We have that
\begin{align} \label{eq_modlhs}
    \| \chi \hd u\|_{H^1_{ \theta, scl}}  
    &\geq
    \| \hd (\chi u)\|_{H^1_{ \theta, scl}} - \| [\chi,\hd] u \|_{H^1_{ \theta, scl}}
\end{align} 
By expression  \eqref{eq_comexp} for the symbol of a commutator we have that 
\[
    \sigma([\chi,\hd]) = -\frac{h}{i}\nabla_y \chi \nabla_\xi \jb{\xi}^{-1} + h^2\mathcal{S}^{-1}.
\] 
It follows then from  Proposition  \ref{weightedPSDO} that
\begin{align} \label{eq_chi_comm}
    \| [\chi,\hd] u \|_{H^1_{\theta,scl}}  \leq C h \| u \|_{L^2_{\theta}}.
\end{align} 
Next we estimate the middle term in \eqref{eq_modlhs}, as follows
\begin{align*}
    \| \hd u \|_{H^1_{ \theta, scl}} 
    &=  \| \jb{r}^\theta \hd u\|_{H^1_{scl}} \\
    &\geq C   \| \hd (\jb{r}^\theta  u)\|_{H^1_{scl}} - \| [\jb{r}^\theta, \hd] u\|_{H^1_{scl}} \\
    &\geq C   \| u \|_{L^2_\theta} - \| [\jb{r}^\theta, \hd] u\|_{H^1_{scl}} 
\end{align*} 
The commutator can be estimated by Lemma \ref{decay_lemma}
\begin{align*}
    \| [\jb{r}^\theta, \hd] u\|_{H^1_{scl}} \leq C h \| \jb{r}^{\theta-1}  u\|_{L^2} 
    \leq C h \| u\|_{L^2_\theta}.
\end{align*} 
We can hence absorb 
the commutator by the first term, when $h$ is small and get 
\begin{align*}
    \| \hd u \|_{H^1_{ \theta, scl}} 
    &\geq C   \| u \|_{L^2_\theta}.
\end{align*} 
It follows from \eqref{eq_modlhs} using the above estimate and \eqref{eq_chi_comm} that
\begin{align*}
    \| \chi \hd u\|_{H^1_{ \theta, scl}}  
    &\geq C   \| u \|_{L^2_\theta} -  C h \| u \|_{L^2_{\theta}} \\
    &\geq C   \| u \|_{L^2_\theta}, 
\end{align*} 
when $h$ is small.
We can thus estimate the left hand side of \eqref{eq_p12} from below as follows
\begin{align} \label{eq_shifted_ce1}
    \frac{C}{\eps} \|r^{\frac{\delta}{2}-1} u\|_{L^2}
    \leq 
    \|P_h(\chi\hd u)\|_{L^2}. 
\end{align}
Splitting the right hand side of \eqref{eq_shifted_ce1} using the basic properties of commutators, 
gives that
\begin{align} \label{eq_shifted_ce2} 
    \|P_h(\chi\hd u)\|_{L^2}
    &\leq
     \| \chi\hd P_h u \|_{L^2}  \nonumber \\
     &+ \| \chi [P_h,\hd] u\|_{L^2}  \\
     &+ h^2\|e^{\varphi_{\eps}/h}[\Delta - \la^2,\chi] (e^{-\varphi_{\eps}/h}\hd u)\|_{L^2} \nonumber.
\end{align}
To obtain the estimate in the statement of the Proposition, we need to show that
the second and third term can be absorbed by the left hand side of  \eqref{eq_shifted_ce1}, 
when $\eps$ is chosen small enough.  This can be done if we can bound these terms 
in the weighted $L^2$-norm.

Writing out the commutator in the third term yields
\begin{align*}
    h^2\|e^{\varphi_{\eps}/h}[\Delta,\chi] (e^{-\varphi_{\eps}/h}\hd u)\|_{L^2}
     &\leq C h^2 \big( \| \Delta \chi \hd u \|_{L^2}  \\
     &+ \| e^{\varphi_{\eps}/h} \nabla \chi \cdot \nabla e^{-\varphi_{\eps}/h} \hd u \|_{L^2}  \\
     &+\| \nabla \chi \nabla (\hd u) \|_{L^2} \big).
\end{align*}
We now find bounds for the terms on the right hand side in the weighted $L^2$-norm.
For the last term we note that $\sigma(h \nabla \hd) \in \mathcal{S}^0$. By Proposition \ref{weightedPSDO}
we know that $h \nabla \hd\colon L^2_\theta \to L^2_\theta$ is continuous
and hence that
\begin{align*}
    h^2\| \nabla \chi \nabla (\hd u) \|_{L^2} \leq Ch \| h\nabla (\hd u) \|_{L^2_\theta} 
    \leq Ch \| u \|_{L^2_\theta}.
\end{align*}
For the two remaining terms, we have that $\sigma(\hd) \in \mathcal{S}^{-1}$. By
Proposition \ref{weightedPSDO} we know that
$\hd\colon L^2_\theta \to L^2_\theta$  is continuous and thus we have in the same way that
\begin{align*}
    h^2(\| \Delta \chi \hd u \|_{L^2}  
    + \| e^{\varphi_{\eps}/h} \nabla \chi \cdot \nabla e^{-\varphi_{\eps}/h} \hd u \|_{L^2})
    &\leq  Ch  \| u \|_{L^2_\theta}.
\end{align*}
Combining the two previous estimates, gives an estimate for the second commutator 
term in \eqref{eq_shifted_ce2}, i.e.
\begin{align*}
    h^2\|e^{\varphi_{\eps}/h}[\Delta,\chi] (e^{-\varphi_{\eps}/h}\hd u)\|_{L^2}
    \leq  C h  \|u\|_{L^2_\theta}.
\end{align*}

It remains to estimate the first commutator term in \eqref{eq_shifted_ce2}. I.e. we want 
to show that
\begin{align} \label{eq_2ndcom}
    \| \chi [P_h,\hd] u\|_{L^2}  
    \leq  C  \| u\|_{L^2_\theta}.
\end{align}
Parts of $P_h$ commute with $\hd$, so that we are left with
\begin{align} \label{eq_shifted_ce3} 
    [P_h,\;\hd] 
                & = [-|d\varphi_\eps|^2, \; \hd] 
                    + 2h[\nabla\varphi_\eps \cdot \nabla, \; \hd] \nonumber\\
                    &\quad- h[\Delta \varphi_\eps,\; \hd]. 
\end{align}
As in the proof of Proposition \ref{carlemaninend}, we utilize the 
asymptotics given by \eqref{eq_phiasym} according to which
$|d\varphi_\eps|^2 = c + O(r^{\theta+\delta/2})$, where $c$ is a constant 
and $\Delta \varphi_\eps = O(r^{-1+\theta+\delta/2})$. This enables us to
apply Lemma \ref{decay_lemma} to the first and third commutator in \eqref{eq_shifted_ce3},
by which we get an improvement in decay, which is crucial.
We get that
\begin{align*}
    \|[|d\varphi_\eps|^2, \; \hd] u \|_{L^2} 
    + \| h[\Delta \varphi_\eps,\; \hd] \|_{L^2}
    &\leq   Ch  \| \jb{r}^{\theta+\delta/2-1} u\|_{L^2} \\
    &\leq   Ch  \| u\|_{L^2_\theta},
\end{align*}
where $C$ independent of $\eps$.
The second commutator term in \eqref{eq_shifted_ce3} is
\begin{align*}
                 2h[\nabla\varphi_\eps \cdot \nabla, \; \hd] 
                 &= 2\nabla\varphi_\eps \cdot  [h \nabla , \; \hd]
                 + 2[\nabla\varphi_\eps \cdot , \; \hd]h\nabla  \\
                 &= [\nabla\varphi_\eps \cdot , \; \hd]h\nabla.
\end{align*}
The asymptotics in \eqref{eq_phiasym} give that 
$\nabla \varphi_\eps = (\gamma_1, \gamma_2) + (b_1, b_2)$, where $\gamma_j$ are
constants and $b_j = O(r^{\theta+\delta/2})$.
The above commutator can be estimated by, applying 
Lemma \ref{decay_lemma} to the components of $\nabla \varphi_\eps$, giving 
\begin{align*}
    \| [ b_j , \; \hd]h \nabla u \|_{L^2}
    &\leq Ch \| \jb{r}^{\theta+\delta/2-1} h\nabla u \|_{H^{-1}_{scl}}\\
    &\leq Ch ( \| h\nabla( \jb{r}^{\theta+\delta/2-1} u) \|_{H^{-1}_{scl}} 
    + \|[\jb{r}^{\theta+\delta/2-1},h\nabla] u\|_{H^{-1}_{scl}})\\
    &\leq Ch( \|  u \|_{L^2_\theta} + h \|\jb{r}^{\theta+\delta/2-2} u\|_{H^{-1}_{scl}}) \\
    &\leq Ch \|  u \|_{L^2_\theta}, 
\end{align*}
for small $h$ and where $C$ does not depend on $\eps$.
We thus see that \eqref{eq_2ndcom} holds.

\end{proof}

To complete the proof of the previous Proposition we need to prove the following Lemma.

\begin{lemma} 
\label{decay_lemma}
Let $\kappa\geq0$ and $b(y) \in C^\infty(\R^2)$ satisfies the estimate
\[ |\partial_y^\beta b| \leq C_{\beta} \jb{y}^{-\kappa-|\beta| },\] Then we have the estimate
\[
    \| [ \hd , b ] u \|_{H^{s+2}_{scl}} \leq C h \| \jb{r}^{-\kappa-1} u\|_{H^{s}_{scl}},
\]
for $u \in C^\infty_0(\R^2)$ and $s \in \R$.
\end{lemma}

\begin{proof} 
    Let $a(\xi) := \sigma(\hd) = \jb{\xi}^{-1}$.
The composition $\op(a) \op(b) $ can be written as
follows
\begin{align*}
   \op(c) u =  \op(a) \op(b) u = (2 \pi)^{-2} \int_{\R^2} e^{iy\cdot \xi} c(y,h\xi;h) \hat u (\xi) \,d\xi,
\end{align*}
where the symbol $c$  can in turn
be written in terms of oscillatory integrals, as
\[
    c(y,\xi;h) =  (2 \pi h)^{-2} \int_{\R^2} e^{i z\cdot \xi/h} K(z;h) b(y+z) \,dz,
\]
where
\[
    K(z;h) := \int_{\R^2} e^{-i z\cdot \eta/h} a(\eta) \,d\eta.
\]
We can split $c$ by the Taylor Theorem as follows
\begin{align*}
    c(y,\xi;h) =  (2 \pi h)^{-2} \int_{\R^2} e^{i z\cdot \xi/h} K(z;h) 
    \big(b(y) + R_1(y,z)\big)  \,dz,
\end{align*}
where $R_1$ is a remainder term given by
\[ R_1(y,z) = \sum_{j=1}^2 z_j \int_0^1  (\partial_jb)(y+\theta z) d\theta.\]
Here $\partial_j b(y)$ denotes the partial derivative of $b$ with respect to the $j$-th variable.\\

One sees easily using the fact 
that\footnote{here $\delta$ is the Dirac delta function} 
$\hat 1 (\eta)= \delta(\eta)$,  that
\begin{align*}
    c(y,\xi;h)  = ab + (2 \pi h)^{-2} \int_{\R^2} e^{i z\cdot \xi/h} K(z;h)  R_1(y,z) \,dz.
\end{align*}
A direct consequence of this is that
\begin{align} \label{eq_comsym}
    \sigma([\op(a),\op(b)]) = (2 \pi)^{-2} \int_{\R^2} e^{i z\cdot \xi} K(hz;h)  R_1(y,hz) \,dz.
\end{align}

Define $\mu := h^{-1}\jb{y}^{1+\kappa} \sigma([\op(a),\op(b)])$. We will now show that $\mu \in \mathcal{S}^{-2}$.
For this we  need to check that condition 
\eqref{eq_symclass} holds, with $m=-2$, i.e. we need to show that
\begin{equation} \label{eq_symcond}
| \partial_y^{\beta} \partial_{\xi}^{\gamma} \mu(y,\xi;h)| 
\leq C_{\beta\gamma}\jb{\xi}^{-2},  
\end{equation}
where $C_{\beta\gamma}$ is independent of $h$. By the assumption on $b$ and
the form of $K$ it suffices show this for the case $\beta=\gamma=0$. 

Splitting the integral in \eqref{eq_comsym} by the triangle inequality into two 
components with the index $j=1,2$ and by integrating by parts, we see that we can estimate
\begin{align*}
    I_j := 
    h \Big| \int_{\R^2} e^{i z\cdot \xi} \int_{\R^2} e^{-i z\cdot \eta} \p_j a(\eta) \,d\eta  
    \; \int_0^1(\partial_j b)(y+\theta hz)d\theta   \,dz\Big| ,
\end{align*}
to get get an estimate for $\sigma([\op(a),\op(b)])$.\\

We use the abbreviation $B_j(y,z) := \int_0^1  (\partial_j b)(y+\theta z) d\theta.$ Then by integrating
by parts, we have that
\begin{align*} 
    I_j &= h \Big | \int e^{i z \cdot \xi} 
    \int \jb{z}^{-2N} e^{-i z \cdot \eta} \jb{D_\eta}^{2N} \p_j a(\eta) \,d\eta \; B_j(y,hz)\, dz \Big | \\ 
    &= h \Big | \int 
    \int \frac{e^{i z \cdot (\xi-\eta )} }{\jb{\xi - \eta}^2 } 
    \jb{D_\eta}^{2N} \p_j a(\eta) \,d\eta \; \jb{D_z}^2 \frac{B_j(y,hz)}{\jb{z}^{2N}}\, dz \Big |.
\end{align*}
By  the Peetre inequality
\begin{align*}
    I_j &\leq h \Big\| \frac{ \jb{D_\eta}^{2N} \p_j a  }{\jb{\xi-\eta}^2 }   \Big\|_{L^\infty(\R^2_\eta)}
    \Big\| \int  e^{-i z \cdot \eta} \jb{D_z}^2 \frac{B_j(y,hz)}{\jb{z}^{2N}}\, dz \Big\|_{L^1(\R^2_\eta)} \\
    &\leq h \jb{\xi}^{-2} \Big\| \jb{\eta}^2 \jb{D_\eta}^{2N} \p_j a   \Big\|_{L^\infty(\R^2_\eta)}
\Big\| \int e^{-i z \cdot \eta} \jb{D_z}^2 \frac{B_j(y,hz)}{\jb{z}^{2N}}\, dz \Big\|_{L^1(\R^2_\eta)}.
\end{align*}
Moreover by the Cauchy-Schwarz inequality we have that
\begin{small}
\begin{align*}
         \Big\|\int  e^{-i z \cdot \eta} \jb{D_z}^2 \frac{B_j(y,hz)}{\jb{z}^{2N}}\, dz \Big\|_{L^1}
         &=
         \Big\| \int  \frac{\jb{D_z}^{2k}}{\jb{\eta}^{2k}} e^{-i z \cdot \eta}
         \jb{D_z}^2 \frac{B_j(y,hz)}{\jb{z}^{2N}}\, dz \Big\|_{L^1} \\
         &\leq 
         \|  \jb{\eta}^{-2k} \|_{L^2}
         \Big\| \int  e^{-i z \cdot \eta} 
         \jb{D_z}^{2k+2} \frac{B_j(y,hz)}{\jb{z}^{2N}}\, dz \Big\|_{L^2}. 
\end{align*}
\end{small}
Thus
{\small 
    \[
        I_j \leq \frac{C_{k,N} h }{\jb{\xi}^2}\|\jb{\eta}^2  \jb{D_\eta}^{2N} \p_j a \|_{L^p(\R^2_\eta)}
        \Big\|\frac{1}{\jb{\eta}^{2k}} \Big\|_{L^2} 
        \Big\| \int_0^1
        \jb{D_z}^{2k+2} \frac{(\partial_j b)(y+\theta hz)}{\jb{z}^{2N}}d\theta \Big\|_{L^2(\R^2_z)} 
    \]
}
The above norms become finite when $N$ and $k$ are large enough.
By the assumption on $b$ it suffices to show that 
\[ 
    \Big \| \int_0^1  \frac{(\partial_j b)(y+\theta hz)}{\jb{z}^{2N}}d\theta \Big\|_{L^2(\R^2_z)}
    \leq C \jb{y}^{-\kappa -1}.
\]
Indeed, apply the assumption on $b$ and the Peetre inequality we have that
$|(\partial_jb)(y+\theta hz)| \leq C\jb{y+ \theta hz}^{-\kappa - 1} \leq
C\jb{y}^{-\kappa-1} \jb{ hz}^{|\kappa +1|}$ so the inequality holds provided
that $N$ is large enough. It follows that
\[
    |\sigma([\op(a),\op(b)])|\leq  Ch \jb{y}^{-\kappa-1} \jb{\xi}^{-2} .
\]
This proves \eqref{eq_symcond} and $\mu = h^{-1}\jb{y}^{1+\kappa}
\sigma([\op(a),\op(b)]) \in \mathcal{S}^{-2}$ which, by Proposition \ref{weightedPSDO},
means that $\op(\mu) : H^{s}_{\delta, scl} \to H^{s+2}_{\delta, scl}$
continuously. Hence
\begin{align*}
    \| [ \op(a),\op(b) ] u \|_{H^s_{scl}}  \leq Ch \| \jb{r}^{-1-\kappa} u\|_{H^{s-2}_{scl}},
\end{align*}
which is what we needed to prove.

\end{proof}

We can combine Proposition \ref{shifted_ce} and \ref{cmpct_ce} to obtain a global estimate.
To handle the perturbed operator $L_{X,V}$, we need to assume that potentials have decay at least 
as fast as the weights on the $L^2$-norms.

\begin{lemma} \label{global_perturbed}
Let $\varphi_\eps$ be given by \eqref{eq_phieps}.
Then for all $V\in x^{1-\frac{\delta}{2}}L^\infty(M_0)$ 
and $X \in x^{1-\frac{\delta}{2}}W^{1,\infty}(M_0,T^*M_0)$ 
there exists an
$h_0>0$, $\eps_0$ and $C>0$  such that for all $0<h<h_0$, $h\ll \eps<\eps_0$
and $u\in e^{-\gamma/x} C^\infty(M_0)$, we have 
\begin{align*}
    \frac{C}{\eps} \|x^{1-\frac{\delta}{2}} u\|_{L^2} 
    \leq 
    \sqrt{h} \|e^{\varphi_{\eps}/h}(L_{X,V}-\la^2)e^{-\varphi_{\eps}/h}u\|_{H^{-1}_{scl}}.
\end{align*}
\end{lemma}
\begin{proof} We first consider only the case when $X=0$ and $V=0$. Let $u\in
    e^{-\gamma/x}C^\infty(M_0)$ and pick $\chi \in C^\infty_0(M_0)$ such that $\chi = 1$
    on the compact set containing all the critical points of $\varphi$ and
    $\supp(1-\chi)$ is contained in the ends. By Propositions \ref{cmpct_ce} we
    have the following estimate for $\chi u$
\begin{small}
\begin{align*}
\frac{Ch}{\eps}\big(\sqrt{h} \| \chi u\|_{L^2} &+ \| d\varphi_\eps \chi u\|_{L^2} 
    + \| h d(\chi u)\|_{H^{-1}_{scl}}  \big) 
    \leq \|e^{\varphi_\eps/h} h^2(\Delta-\lambda^2) e^{-\varphi_\eps/h}\chi u\|_{H^{-1}_{scl}}\\
&\leq \|\chi e^{\varphi_\eps/h} h^2(\Delta-\lambda^2) e^{-\varphi_\eps/h} u\|_{H^{-1}_{scl}} 
+ \|e^{\varphi_\eps/h}[h^2\Delta,\chi]e^{-\varphi_\eps/h} u\|_{H^{-1}_{scl}}.
\end{align*}
\end{small}
A limiting argument shows that Proposition \ref{shifted_ce} can be applied to
smooth function with exponential decay. Therefore, we have the following
estimates for
$(1-\chi)u$
{\small \begin{eqnarray*}
\frac{Ch}{\eps} \| x^{1-\delta/2} (1-\chi) u\|_{L^2}
\leq \|e^{\varphi_\eps/h} h^2(\Delta-\lambda^2) e^{-\varphi_\eps/h}(1-\chi) u\|_{H^{-1}_{scl}}\\
\leq \|(1-\chi) e^{\varphi_\eps/h} h^2(\Delta-\lambda^2) e^{-\varphi_\eps/h}u\|_{H^{-1}_{scl}}
+ \|e^{\varphi_\eps/h}[h^2\Delta, \chi]e^{-\varphi_\eps/h} u\|_{H^{-1}_{scl}}.
\end{eqnarray*}}
Adding these two inequalities together we obtain
{\small\begin{eqnarray}
\label{glued estimate}
\frac{Ch}{\eps}\big( \sqrt{h} \| x^{1-\delta/2} u\|_{L^2} + \| d\phi_\eps \chi u\|_{L^2} +
\| x^{1-\delta/2} (1-\chi) u\|_{L^2}
+ \| hd(\chi u)\|_{H^{-1}_{scl}}\big)\leq\\\nonumber
\| e^{\varphi_\eps/h} h^2(\Delta-\lambda^2) e^{-\varphi_\eps/h}
u\|_{H^{-1}_{scl}} + \|e^{\varphi_\eps/h}[h^2\Delta, \chi]e^{-\varphi_\eps/h}
u\|_{H^{-1}_{scl}}.
\end{eqnarray}}
The next step is to absorb the commutator term on the right-side. To this end
we first observe that on the left-side
\[
    \|x^{1-\delta/2} ud\varphi_\eps\|_{L^2} \leq \| \chi u d\varphi_\eps\|_{L^2} 
    + \|x^{1-\delta/2} (1-\chi) u\|_{L^2}
\]
 while on the right-side 
\[ 
    \|e^{\varphi_\eps/h}[h^2\Delta, \chi]e^{-\varphi_\eps/h} u\|_{H^{-1}_{scl}}
    \leq h\|\tilde \chi u\|_{L^2}
\]
for some smooth cut-off $\tilde \chi \in C^\infty_0(M_0)$ which is equal to $1$
on $\supp(d\chi)$ but supported away from the critical points of $\varphi$.
These two inequalities allows one to absorb the
commutator term on the right-side of \eqref{glued estimate}, when taking
$\eps>0$ small enough, to obtain
{\Small\begin{eqnarray*}
\frac{Ch}{\eps}\big( \sqrt{h} \| x^{1-\delta/2} u\|_{L^2} 
+ \|x^{1-\delta/2} ud\varphi_\eps\|_{L^2} + \| hd(\chi u)\|_{H^{-1}_{scl}}\big)
\leq \| e^{\varphi_\eps/h} h^2(\Delta-\lambda^2) e^{-\varphi_\eps/h}
u\|_{H^{-1}_{scl}}.
\end{eqnarray*}}
We now replace the Laplacian by the more general operator $L_{X,V}$. Observe
that $L_{X,V} -\Delta = 2\langle X, d\cdot\rangle + Q$ for some $Q\in
e^{-\gamma/x} L^\infty$ and the zeroth order term can be absorbed to the
left-side. Therefore
\begin{align*}
\frac{Ch}{\eps}&\big( \sqrt{h} \| x^{1-\delta/2} u\|_{L^2} 
+ \|x^{1-\delta/2} ud\varphi_\eps\|_{L^2} + \| hd(\chi u)\|_{H^{-1}_{scl}}\big) \\
&\leq
\| e^{\varphi_\eps/h} h^2(L_{X,V}-\lambda^2) e^{-\varphi_\eps/h}
u\|_{H^{-1}_{scl}} + h \|e^{\varphi_\eps/h}\langle X,h d
(e^{-\varphi_\eps/h}u)\rangle \|_{H^{-1}_{scl}}.
\end{align*}
Again we need to absorb the last term on the right-side. This is done by first
observing that
\begin{align*}
    h \|e^{\varphi_\eps/h}\langle X,h d
        (e^{-\varphi_\eps/h}u)\rangle \|_{H^{-1}_{scl}} 
        &\leq h \big\| |X| |d\varphi_\eps| u \big\|_{L^2} + h \|\langle X ,hd(\chi u)\rangle\|_{H^{-1}_{scl}} \\
        &\quad + h \|\langle X, hd(1-\chi) u\rangle\|_{H^{-1}_{scl}}\\ 
        &\leq h \|
        x^{1-\delta/h} u d\varphi_\eps\| + h \| hd(\chi u)\|_{H^{-1}_{scl}}  \\
        &\quad + h \|x^{1-\delta/2} (1-\chi)u\|_{L^2}.
\end{align*}
One sees then that the extra term can indeed be absorbed into the left-side by taking $\eps>0$ small enough.
\end{proof}

We can utilize the above estimates to obtain an existence result,
which is needed when constructing the CGO solutions.

\begin{lemma}
\label{solvability1}
Let $\delta \in (0,1)$, $V\in x^{1-\frac{\delta}{2}}L^\infty(M_0)$, 
$X \in x^{1-\frac{\delta}{2}}W^{1,\infty}(M_0,T^*M_0)$ and
$\varphi_\eps$ as in \eqref{eq_phieps}. 
For all $f\in L^2(M_0)$ and all $h>0$ small enough, there exists a solution
$u\in L^2(M_0)$ to the equation
\begin{equation}\label{solvab_plain}
    e^{\varphi_\eps/h}(L_{X,V}-\la^2) e^{-\varphi_\eps/h}u = x^{1-\frac{\delta}{2}}f
\end{equation}
satisfying
\[
    \|u\|_{L^2} + h \|du\|_{L^2} \leq C \sqrt{h} \|f\|_{L^2}.
\]
\end{lemma}

\begin{proof} 
Let $L:= e^{\varphi_\eps/h}(L_{X,V} - \lambda^2)e^{-\varphi_\eps/h}$ and consider the linear space 
\[{\cal H} := \{ L^* v \mid v \in C^\infty_0(M_0)\}.\]
Define a linear operator $T: {\cal H} \to \C$ by $T (L^*v) := \big( x^{1-\delta/2} v , f \big)_{L^2}$. 
Lemma \ref{global_perturbed} applies also to $L^*$, which shows that $T$ is well defined. 
Observe that $\operatorname{Dom}(T)$ is a linear subspace of $H^{-1}_{scl}(M_0)$.
By Lemma \ref{global_perturbed} again, one has that
{\small\begin{align} \label{eq_Tbnd}
    |T(L^*v)| &= \big|\big( x^{1-\delta/2} v , f \big)_{L^2}\big|\leq \|  x^{1-\delta/2} v\|_{L^2} \|f\|_{L^2}\leq C \eps \sqrt{h} \| L^*v\|_{H^{-1}_{scl}}\|f\|_{L^2}. 
\end{align}}
The map $T$ is hence bounded on the subspace ${\cal H}$ in the $H^{-1}_{scl}(M_0)$ norm. By the Hahn-Banach Theorem this map extends to a bounded linear functional on $H^1_{scl}$ with the same norm, which we still denote by $T$. By duality there
exists a $u \in H^{1}_{scl}(M_0)$, such that $T(w) = \langle u, w\rangle$ for all $w\in H^{-1}_{scl}(M_0)$ where $\langle\cdot, \cdot\rangle$ denotes the duality between $H^{1}_{scl}$ and $H^{-1}_{scl}$. Furthermore $u$ satisfies the estimate $\|u\|_{H^1_{scl}}  = \|T\|_{(H^{-1}_{scl})^*} \leq C\eps \sqrt{h} \|f\|_{L^2}$.
We then have that for all $v\in C^\infty_0(M_0)$,
 \[ \big(x^{1-\delta/2} v, f\big)=T (L^*v)= (u, L^*v)\]
and this is precisely the statement that $u$ is a weak solution of $L u = x^{1-\delta/2} f$.
\end{proof}

Later we conjugate $L_{X,V}$ with an additional function $F_A$, of the 
form specified at the end of Section \ref{Boundary Identifiability at Infinity}.
The functions $F_A$ are in particular smooth non-vanishing functions on $M_0$, which has the expression $F_A  = e^{i\alpha} ( 1+ t)$ with $\alpha$ given by Proposition \ref{constructing alpha}, $t$ bounded uniformly away from $-1$ and in the space $e^{-\gamma/x} W^{1,\infty}(M_0)$ for all $\gamma >0$. The following Proposition gives a solvability
result in terms the additional conjugation.

\begin{proposition}
\label{semiclassical solvability}
Let $V\in x^{1-\frac{\delta}{2}}L^\infty(M_0)$, 
$X \in x^{1-\frac{\delta}{2}}W^{1,\infty}(M_0,T^*M_0)$ and
let $f\in x^J L^2$ for some $J\in \R$. There exists solutions $w\in H^1_{loc}$ to the equation
\begin{align}\label{solvab_FA}
    e^{-\Phi/h}F_A^{-1} (L_{X,V} -\lambda^2) e^{\Phi/h} F_A w = f,
\end{align}
where $\Phi$ is as in \eqref{eq_Phi},
which satisfies, the estimate
\[
    \| e^{\varphi_0 / \eps} w \|_{L^2} + h \|e^{\varphi_0 /\eps} dw \|_{L^2} 
    \leq C \sqrt{h} \| x^{-J} f\|_{L^2},
\]
where $\varphi_0$ is as required in definition \eqref{eq_phieps}.
\end{proposition}

\begin{proof}
By the assumption on the form of $F_A$ it suffices to show this for $F_A = e^{i\alpha}$. Let $f \in x^J L^2$. Since $F_A e^{-\Re(i \alpha)}$ is bounded and 
that $e^{\varphi_0/\eps}$ decays faster than any polynomial in $x$, we have that
\[
e^{(\varphi-\varphi_\eps)/h}F_A f = F_A e^{-\Re(i \alpha)} e^{\varphi_0/\eps} f\in x^{1-\frac{\delta}{2}}L^2(M_0).
\]
By Lemma \ref{solvability1} there is a solution $u$ to the equation
\begin{align*}
e^{-\varphi_\eps/h} (L_{X,V} -\lambda^2) e^{\varphi_\eps/h} u = e^{(\varphi-\varphi_\eps)/h}F_A e^{i\psi/h}f.
\end{align*}
Define $v:= e^{(\varphi_\eps-\varphi)/h}F_A^{-1} u = e^{\Re(i\alpha)}F_A^{-1} e^{-\varphi_0/\eps}u$.
It follows that $v$ solves
\begin{align*}
    e^{-\varphi/h}F^{-1}_A (L_{X,V} -\lambda^2) e^{\varphi/h}F_A v = e^{i\psi/h}f.
\end{align*}
The norm estimate of Lemma \ref{solvability1}  gives 
furthermore that 
\begin{align}\label{eq_upper}
    \| u \|_{L^2}  +h  \| du \|_{L^2}
    \leq C \sqrt{h} \| e^{(\varphi-\varphi_\eps)/h}F_A f \|_{L^2}
    \leq C \sqrt{h} \| x^{-J} f \|_{L^2},
\end{align}
where the second inequality is obtained from the fact that multiplication by
$e^{(\varphi-\varphi_\eps)/h}F_A x^J = F_A e^{-\Re(i \alpha)} e^{\varphi_0/\eps} x^J \in L^\infty(M_0)$
is $L^2$-continuous on $M_0$.
Next we estimate the right hand side of \eqref{eq_upper} from below. 
Firstly 
\begin{align} \label{eq_vest}
    \| u \|_{L^2}  =  \| F_A e^{-\Re(i\alpha)} e^{\varphi_0/\eps}v \|_{L^2} 
    \geq C  \| e^{\varphi_0/\eps} v\|_{L^2},
\end{align}
since $e^{\Re(i \alpha)} /F_A \in L^\infty(M_0)$.
Expanding the derivative and use the assumption that $F_A = e^{i\alpha}$ gives that
\begin{align*}
    du = F_A e^{-\Re(i\alpha)} e^{\varphi_0/\eps} \big( (d\varphi_0/\eps - i\Re(d\alpha)) v + dv \big). 
\end{align*}
Since $F_A^{-1} e^{\Re(i\alpha)} \in L^\infty(M_0)$, we have that
\begin{align*}
    \| du \|_{L^2} 
    &\geq
    C \big( \| e^{\varphi_0/\eps} dv \|_{L^2} 
    - \| e^{\varphi_0/\eps} \big( d\varphi_0/\eps - i\Re(d\alpha)  \big) v  \|_{L^2} \big) \\
    &\geq
    C \big( \| e^{\varphi_0/\eps} dv \|_{L^2} - \| e^{\varphi_0/\eps}  v   \|_{L^2}  \big),
\end{align*}
where in the second step, we used that $d\varphi_0/\eps-i\Re(d\alpha) \in L^\infty(M_0)$,
which holds because of \eqref{eq_phiasym} and because of the expression of $\alpha$ given in Proposition \ref{constructing alpha}.
This together with \eqref{eq_vest}, gives 
\begin{align} \label{eq_lower}
     h^{-1}\| e^{\varphi_0/\eps}  v \|_{L^2}  
    + \| e^{\varphi_0/\eps}  dv \|_{L^2}  
    \leq C ( h^{-1}\| u \|_{L^2}  +  \| du \|_{L^2} ),
\end{align} 
when $h$ is small. From \eqref{eq_upper} we get that
\[
    \| e^{\varphi_0/\eps}  v \|_{L^2}  + h \| e^{\varphi_0/\eps} dv \|_2 \leq C \sqrt{h} \| x^{-J} f \|_{L^2}.
\]
Finally setting $w:= e^{-i\psi/h}v$, we see that $w$ solves
\begin{align*}
    e^{-\Phi/h}F^{-1}_A (L_{X,V} -\lambda^2) e^{\Phi/h}F_A w = f,
\end{align*}
and that we have the estimate of the claim.

\end{proof}

\end{section}

\begin{section}{Scattering by $L_{X,V}$ on Surfaces with Euclidean Ends}
    \label{sec_scat}
In this section we construct the scattering matrix through the use of the
Poisson operator for the operator $L_{X,V}$ on surfaces with Euclidean ends.
Furthermore we will show that the range of the Poisson operator is dense in
some suitably defined exponentially weighted solution spaces:
\begin{proposition}
\label{poisson kernel}
There exists an operator $P_{X,V}(\lambda) : C^\infty(\partial M_0) \to
x^{-\tau} H^1(M_0)$ satisfying for all $f_+ \in C^\infty(\partial M_0)$ there
exists a unique $f_-\in C^\infty(\partial M_0)$ such that 
{\Small\begin{eqnarray}
\label{poisson expansion}
P_{X,V}(\lambda) f_+ - (x^{1/2}  e^{\lambda/x}f_+ + x^{1/2} e^{-\lambda/x} f_-
)\in L^2(M_0),\ \ \ (L_{X,V}- \lambda^2)P_{X,V}(\lambda) f_+ = 0.\end{eqnarray}
}
We define the scattering matrix $S_{X,V}(\lambda)$ by $S_{X,V}(\lambda) f_+ :=f_-$.
\end{proposition}
\begin{proposition}
\label{density}
Let $0< \gamma<\gamma'<\gamma_0$. If $X\in e^{-\gamma_0/x} L^\infty$ and $V\in
e^{-\gamma_0/x} L^\infty$ the set \[
\{P_{X,V}(\lambda) f \mid f\in C^\infty (\partial M_0)\}\]
is dense in the null space of $L_{X,V} -\lambda^2$ in $e^{\gamma/x} L^2$ with
respect to the $e^{\gamma'/x} L^2$ topology.
\end{proposition}
We first define the free resolvent $R_0(\lambda) : L^2 \to H^2$ on $\R^2$ for $\lambda$ on the lower half of the complex plane. If $A>0$ then for all $\gamma > A$ this resolvent extends as a holomorphic family of operators $R_0(\lambda) : e^{-\gamma/x} L^2 \to e^{\gamma/x} H^2$ as $\lambda$ vary over the set $\{\lambda \mid \im(\lambda) <A, \lambda \notin i\R^+\cup 0\}$. Direct computation also yields that for all $\tau>1/2$ one has $R_0(\lambda) : x^\tau L^2 \to x^{-\tau} H^2$ when $\lambda$ lies on the positive real axis. This fact is usually stated in weighted $L^2$ spaces but the Sobolev estimate can be obtained by writing
\[(d^*d + 1)R_0(\lambda) = Id + (\lambda^2 + 1) R_0(\lambda).\]
We generalize this statement for the operator $L_{X,V}$ on the surface $M_0$:

\begin{lemma}
\label{resolvent}
If $A>0$ then for all $\gamma > A$ the resolvent $R_{X,V}(\lambda) := (L_{X,V} - \lambda^2)^{-1}$ is defined as a meromorphic family of operators mapping $e^{-\gamma/x} L^2 \to e^{\gamma/x} L^2$ over the set $\{\lambda \mid \im(\lambda) <A, \lambda \notin i\R^+\cup 0\}$. Furthermore, if $\lambda \in \R^+$ is not a pole of $R_{X,V}(\lambda)$ then it is a bounded map from $x^{\tau} L^2 \to x^{-\tau} H^1$ for any $\tau>1/2$. 
\end{lemma}
\begin{proof}
 We let $\chi \in C^\infty_0(M_0)$ be a smooth function such that $1-\chi$ is supported near $E_j$. We let $\chi_0, \chi_1\in C^\infty_0(M_0)$ be smooth functions such that $\chi_0 = 1$ on the support of $\chi$ and $1-\chi_1 = 1$ on the support of $1-\chi$. We observe that if we chose $\lambda_0$ to have a large negative imaginary part, then for the parametrix
\[E(\lambda) := (1- \chi_1)R_0(\lambda) (1-\chi) + \chi_0 R_0(\lambda_0) \chi\] 
we have $(L_{X,V} - \lambda^2) E(\lambda) = I + K(\lambda)$ where $K(\lambda) : x^\tau L^2(M_0) \to x^\tau H^1(M_0)$ is given by
{\Small \[K(\lambda) = ([\Delta_{g_0}, \chi_1] - (\lambda^2-\lambda_0^2)\chi_1) R_0(\lambda_0) \chi -[\Delta_{g_0}, \chi_0] R_0(\lambda) (1-\chi) + (X^\sharp +\star dX+ |X|^2+V)E(\lambda) \]}and $X^\sharp$ denotes differentiation with respect to the vector field obtained by raising the index on the 1-form $X$. By the mapping properties of \[R_0(\lambda) : x^{\tau}L^2(\R^2) \to x^{-\tau}H^2(\R^2),\  R_0(\lambda_0) : L^2(M_0) \to H^2(M_0),\] and the super-exponential rates of decay of $X$ and $V$, we have that $K(\lambda)$ is a holomorphic family of compact operators from $e^{-\gamma/x} L^2$ to itself. If $\lambda = \lambda_0$ has a large negative imaginary part, then $I+ K(\lambda)$ is invertible by Neumann series. Therefore, by the analytic Fredholm theorem  $(I + K(\lambda))^{-1}$ is a meromorphic family of operators from $e^{-\gamma/x} L^2$ to itself as $\lambda$ varies over the region $ \{\lambda \mid \im(\lambda) <A, \lambda \notin i\R^+\cup 0\}$. Setting $R_{X,V}(\lambda) := E(\lambda) (1+ K(\lambda))^{-1}$ proves the portion of the Lemma for the exponentially weighted $L^2$ spaces.

For the resolvent acting on $x^\tau L^2$, we need to show that $1 + K(\lambda)$ is invertible on $x^\tau L^2$ for $\tau> 1/2$. Similar argument as before shows that $K(\lambda)$ is compact on $x^\tau L^2$ and therefore the invertibility of $1 + K(\lambda)$ at a given $\lambda \in \R^+$ can be deduced from the triviality of its null-space. Indeed, if $\lambda$ is not a pole of the resolvent $R_{X,V}(\lambda)$ acting on $e^{-\gamma/x}L^2$, then $1 + K(\lambda)$ is invertible on $e^{-\gamma/x} L^2$. Suppose $u\in x^\tau L^2$ is in the null-space of $1 + K(\lambda)$ then it is actually an element of $e^{-\gamma/x}L^2$ by the decay properties of the coefficients in $K(\lambda)$. As $1+K(\lambda)$ is invertible on $e^{-\gamma/x}L^2$, we have that $u = 0$. Therefore, $R_{X,V}(\lambda) = E(\lambda) (1 + K(\lambda))^{-1} $ is a resolvent mapping $x^\tau L^2 \to x^{-\tau } H^1$ when $\lambda \in \R^+$ is not a pole. 
\end{proof}
It is well-known (\cite{MelStanford}) that for all $f\in e^{-\gamma/x} L^2(\R^2)$ the free resolvent has asymptotic given by
\[R_0(\lambda) f - x^{1/2} e^{-i\lambda/x} v \in L^2(\R^2)\]
for some smooth function $v \in C^\infty(S^1)$. By the construction of $E(\lambda)$ and $R_{X,V}(\lambda)$ this gives the expansion
\begin{eqnarray}
\label{asymp of resolvent}
E(\lambda) f - x^{1/2} e^{-i\lambda/x} v \in L^2(M_0),\ \ \ R_{X,V}(\lambda) f - x^{1/2} e^{-i\lambda/x} v' \in L^2(M_0)
\end{eqnarray}
for some $v,v' \in C^\infty(\partial M_0)$.

We would like to prove that the resolvent has no poles on $\R^+$. Following the exposition of \cite{MelStanford} we first prove that 
\begin{lemma}
\label{pole characterization}
The poles of resolvent $R_{X,V}(\lambda)$, are precisely the values $\lambda$ for which there exists a nontrivial solution $u\in x^{-\tau}H^1(M_0)$ of the equation $(L_{X,V} - \lambda^2) u = 0$ satisfying $ u - x^{1/2} e^{-i\lambda/x} v \in L^2(M_0)$ for some smooth function $v\in C^\infty(\partial M_0)$.

\end{lemma}
\begin{proof}
If $\lambda'$ is a pole of $R_{X,V}(\lambda)$ then it must be a pole of $(1 + K(\lambda))^{-1}$ as the parametrix $E(\lambda)$ is holomorphic. Therefore, there exists $f\in e^{-\gamma/x} L^2$ for which $(1 + K(\lambda))^{-1} f$ has a pole at $\lambda'$ with residue $u' \in e^{-\gamma/x}L^2$. Using the fact that $(1 + K(\lambda))(1 + K(\lambda))^{-1} f = f$ we have that $u' = - K(\lambda') u'$. Therefore, if we set $u := E(\lambda') u'$ then $(L_{X,V} - \lambda'^2) u = (1+ K(\lambda')) u' = 0$ and the asymptotic of $u$ can be derived from \eqref{asymp of resolvent}. 
\end{proof}
We now show that the embedded eigenvalue obtained in Lemma \ref{pole characterization} must be trivial. To this end we first derive the boundary pairing identity
\begin{lemma}
\label{boundary pairing}
For $\lambda >0$ and $X,V\in e^{-\gamma/x}L^\infty(M_0)$, if $u_\pm\in x^{-\tau} H^1(M_0)$ for some $\tau>\demi$ and $(L_{X,V} - \lambda^2)u_\pm \in x^\tau L^2(M_0)$ with 
\[u_\pm - x^{1/2}e^{i\lambda/x} f_{\pm +} - x^{1/2} e^{-i\lambda/x} f_{\pm -}\in L^2(M_0)\]
then we have the integral identity
\[\langle (L_{X,V} - \lambda^2)u_+,u_-  \rangle - \langle u_+, (L_{X,V} -\lambda^2)u_- \rangle = 2i\lambda\int_{\partial M_0} (f_{++} \bar f_{-+} - f_{+-} \bar f_{--} )\]
where the volume form on $\partial M_0$ is induced by the metric $x^2 g_0\mid_{T\partial M}$.
\end{lemma}
\begin{proof} It suffices to prove this for $\Delta_{g_0}$ in place of $L_{X,V}$ and use the fact that $L_{X,V} - \Delta_{g_0}$ is a symmetric first order differential operator with super-exponential decaying coefficients. 

If $u_\pm \in x^{-\tau}H^1$ with $(\Delta_{g_0} -\lambda^2) u_\pm \in
x^{\tau}L^2$ and \[r_\pm := u_\pm - x^{1/2}e^{i\lambda/x} f_{\pm +} - x^{1/2}
e^{-i\lambda/x} f_{\pm -}\in L^2(M_0)\]
then one can deduce that $r_\pm\in H^2(M_0)$. Therefore, if for $\epsilon >0$
small we denote $\langle f,g\rangle_{x>\epsilon}:=\int_{\{x> \epsilon\}}f\bar g
{\rm dvol}_{g_0}$, we have
\begin{align*}
    \langle (\Delta_{g_0} - \lambda^2) &u_+,u_-
    \rangle_{x>\epsilon} - \langle u_+, (\Delta_{g_0} -\lambda^2)u_- \rangle_{x>\epsilon}  \\
    &= \int_{ \{x =\epsilon\}} \bar u_-\partial_\nu u_+ - u_+ \partial_\nu \bar u_- \\
    &= I_\epsilon + \epsilon^{-1/2}\int_{ \{x =\epsilon\}}
(a_+ (r_++\partial_\nu r_+)+ a_-(r_- + \partial_\nu r_-)) 
\end{align*}
where
$\lim\limits_{\epsilon\to0}I_\epsilon =2i\lambda\int_{\partial M_0} (f_{++}
\bar f_{-+} - f_{+-} \bar f_{--})$ with the volume form induced by the metric
$x^2 g_0\mid_{T\partial M_0}$ and $a_\pm$ are $L^\infty(M_0)$ functions. As
$r_\pm \in H^2(M_0)$ we can deduce that there exists a sequence of
$\epsilon_j\to 0$ such that \[\int_{ \{x =\epsilon\}}  |(a_+ (r_++\partial_\nu
r_+)+ a_-(r_- + \partial_\nu r_-))| \leq \epsilon_j.\] 
Taking this sequence $\epsilon_j\to 0$ and use the fact that $ (u_\pm)
(\Delta_{g_0} -\lambda^2)u_\mp \in L^1(M_0)$ by assumption allows us to arrive
at the desired integral identity. \end{proof}
We are now in a position to show that the embedded function $u = x^{1/2}
e^{i\lambda/x} v + L^2(M_0)$ constructed in Lemma \ref{pole characterization}
is trivial when $\lambda\in \R^+$ by repeating an argument in \cite{PSU}.
Indeed, by setting $u_+ = u_- = u$ in Lemma \ref{boundary pairing}, we see that
$v= 0$ and therefore $u\in H^2(M_0)$. Let $\chi\in C^\infty_0(M_0)$ be a smooth
compactly supported function such that $1-\chi$ is only supported in the
Euclidean ends and define $u_\chi := (1-\chi) u$ to be the $H^2$ function
defined on the disjoint union of finitely many copies of $\R^2$. From the
super-exponential decay of the coefficients of $L_{X,V}$, we can use
Paley-Weiner to conclude that $(|\xi|^2 - \lambda^2)\hat u_\chi$ extends to a
holomorphic function $g(\xi + i \eta)$ on $\C^2$ which satisfies the bound
\[\sup\limits_{|\eta| \leq \gamma}\|g(\cdot + i\eta)\|_{L^2} \leq C_\gamma,\ \ \ \forall \gamma >0.\]
The fact that $\hat u_\chi \in L^2$ forces $g$ to vanish on the real variety
$\{\xi\in\R^2\mid \xi\cdot\xi -\lambda^2 = 0\}$ and therefore vanish on the
complex codimension one variety $\{\zeta\in \C^2\mid \zeta\cdot\zeta -\lambda^2
=0\}$ (see proof of Lemma 2.5 \cite{PSU}). One sees then that for all
multi-indices $\beta$ with $|\beta|\leq 2$ the function $\xi^\beta\hat u_\chi$
extends to a holomorphic function on $\C^2$ which satisfies the bound
\[\sup\limits_{|\eta| \leq \gamma}\| (\cdot + i\eta)^\beta\hat u_\chi(\cdot +
i\eta)\|_{L^2} \leq C_\gamma,\ \ \ \forall \gamma >0.\]
Paley-Weiner then shows that $u\in e^{-\gamma/x}H^2(M_0)$ for all $\gamma >0$.
Applying the Carleman estimate in Proposition \ref{global_perturbed} shows that
$u = 0$.

A direct consequence of this discussion in conjunction with Lemma \ref{pole characterization} yields the following 
\begin{cor}
\label{no poles, no L2 solutions}
There does not exist nontrivial solutions to \[(L_{X,V} -\lambda^2) u_\pm = 0\] of the form $u_\pm = x^{1/2} e^{\pm i\lambda/x}v_\pm + L^2(M_0)$ for some $v_\pm \in C^\infty(\partial M_0)$. Furthermore, the poles of the resolvent $R_{X,V}(\lambda)$ does not lie on the positive real axis.
\end{cor}

\noindent{\bf Proof of Proposition \ref{poisson kernel}}   We set \[P_{X,V}(\lambda) := (1-\chi) P_0(\lambda) - R_{X,V}(\lambda) (L_{X,V} - \lambda^2) (1-\chi)P_0(\lambda)\] where $P_0(\lambda)$ is the free Poisson kernel on $\R^2$. The asymptotic expansion of the operator $P_{X,V}(\lambda)$ is then given by \eqref{asymp of resolvent} and the expansion for $P_0(\lambda)$. The uniqueness of the expansion in \eqref{poisson expansion} comes from Corollary \ref{no poles, no L2 solutions}.\qed\\
We are now in a position to show that the range of the Poisson operator is dense in the solution space of exponentially growing solutions.\\
\noindent{\bf Proof of Proposition \ref{density}} 
Let $w\in e^{-\gamma'/x} L^2$ be orthogonal to the range of $P_{X,V}(\lambda)$ so that $\langle w, P_{X,V}(\lambda)f_+\rangle = 0$ for all $f_+\in C^\infty(\partial M_0)$. We need to show that $\langle u , w \rangle = 0$ for all $u\in e^{\gamma/x}L^2$ such that $(L_{X,V}-\lambda^2) u = 0$. 

To this end consider the function $v := R_{X,V}(\lambda) w = x^{1/2} e^{i\lambda/x} f + L^2(M_0)$ for some $f\in C^\infty(\partial M_0)$. Applying integral identity in Lemma \ref{boundary pairing} with $u_+ = P_{X,V}(\lambda)f_+$ and $u_- = v$ we see that $\int_{\partial M_0} f_+ \bar f = 0$ for all $f_+\in C^\infty(\partial M_0)$. This means that $v$ is an $L^2$ solution to $(\Delta_{g_0}-\lambda^2)v \in e^{-\gamma'/x}L^2$. If we choose smooth cutoff $\chi\in C^\infty_0(M_0)$ such that $1-\chi$ is only supported in the Euclidean ends, this would mean that $v_\chi := (1-\chi)v\in L^2$ solves $(\Delta-\lambda^2)v_\chi \in e^{-\gamma'/x}L^2$ in $\R^2$. Repeating the argument made in proving Corollary \ref{no poles, no L2 solutions} we see that $v_\chi$ (and therefore $v$) is an element of $e^{-\gamma'/x} H^2$.

Now let $u\in e^{\gamma/x}L^2$ such that $(L_{X,V}-\lambda^2) u = 0$. We may write 
{\small\[\langle w, u\rangle = \langle (L_{X,V}-\lambda^2) R_{X,V}(\lambda) w, u\rangle =  \langle (L_{X,V}-\lambda^2) v, u\rangle = \langle v, (L_{X,V} -\lambda^2)u\rangle = 0\]}
where the integration-by-parts performed in the last step is permitted since $v \in  e^{-\gamma'/x} H^2$ for some $\gamma' > \gamma>0$. \qed
\end{section}

\begin{section}{Boundary Identifiability at Infinity} \label{sec_bndry_id}
\label{Boundary Identifiability at Infinity}
\begin{proposition} 
\label{exponentially decaying conjugating factor}
Let $\alpha_j \in x^{-J} H^{4+\eps_0}$, $2>J>1$, be solutions to $\bar\partial\alpha_j = A_j : = \pi_{0,1}X_j \in e^{-\gamma/x} H^{3+\eps_0}(M_0)$ constructed in Proposition \ref{constructing alpha}. Assuming that $S_{X_1, V_1}(\lambda) = S_{X_2, V_2}(\lambda)$, there exists a non-vanishing holomorphic function $\Psi$ satisfying \begin{eqnarray}
\label{boundary condition of Psi}\Psi - e^{i(\alpha_1 -\alpha_2)} \in e^{-\gamma/x} H^{3+\eps_0}(M_0)\end{eqnarray}
for all $ \gamma >0$.
\end{proposition}
We will split this into several  Lemmas. In all of them we assume without stating that $S_{X_1, V_1}(\lambda) = S_{X_2, V_2}(\lambda)$.
\begin{lemma}
\label{rapid decay means exponential decay}
Let $f\in H^m_b(M_0)$ be a function satisfying \[\bar\partial f \in   e^{-\gamma /x} H^m_b(M_0; ^bT^*_{0,1}M_0)\] for all $\gamma >0$. Suppose $f\in x^{J} H^m_b(M_0)$ for all $J \in \R$, then $f\in  e^{-\gamma /x} H^m_b(M_0)$ for all $\gamma >0$.
\end{lemma}
\noindent{\bf Proof} By localizing $f$ with cutoff functions near the Euclidean ends $E_j$ and arguing each individual ends separately, we may assume without loss of generality that $M_0 = \cc$ on which we use the standard variable $z$. Denote by 
\begin{eqnarray}
\label{f decays superexponentially}
u:= \bar\partial f \in e^{-\gamma |z|} H^m(\cc),\ \ \forall\gamma >0.
\end{eqnarray}
It suffices to prove that $f\in e^{-\gamma |z|} L^2(\cc)$ as the general Sobolev space result follows by considering $f_j := \partial_{x_j} f$ which satisfies $\bar\partial f_j = \partial_{x_j} u \in e^{-\gamma |z|} H^{m-1}(\cc)$.

Taking Fourier Transform of \eqref{f decays superexponentially} we have that since $f \in |z|^{-J} H^m$ for all $J \in \R$,
\begin{eqnarray}
\label{condition at zero}
(\xi_1 + i\xi_2)\hat f(\xi) = \hat u(\xi), \ \ \ \ \hat f \in C^\infty(\R^2)
\end{eqnarray}
which gives us a condition at the origin that will be useful later. By Paley-Weiner $\hat u(\xi) = \hat u(\xi_1, \xi_2)$ extends to be a holomorphic function on $\C^2$ of two complex variables $\hat u(\zeta_1,\zeta_2)$ with $\zeta_j = \xi_j + i\eta_j \in \C$ with $\xi_j = Re(\zeta_j)$ and $\eta_j = Im(\zeta_j)$ (sometimes we write $\hat u(\zeta_1, \zeta_2) = \hat u(\xi + i\eta)$). Furthermore, it satisfies, by \eqref{f decays superexponentially} and Paley-Wiener,
\begin{eqnarray}\label{FT of f estimate}
\sup\limits_{|\eta| \leq \gamma} \| \hat u(\cdot + i \eta)\|_{L^2}^2 <\infty\ \ \ \ \forall \gamma \geq 0.
\end{eqnarray}

We will prove that $\hat u(\zeta_1, \zeta_2)$ has power series expansion around the origin of the form
\begin{eqnarray}
\label{desired expansion of f}
\hat u(\zeta_1, \zeta_2) = \sum\limits_{j = 1, k = 0}^\infty c_{j,k} \big( \zeta_1 + i \zeta_2\big)^j ( \zeta_1 - i \zeta_2\big)^k .
\end{eqnarray}
Notice that the index $j$ starts at $1$ rather than $0$. If \eqref{desired
expansion of f} holds then $\hat f(\xi_1,\xi_2)$ would 
by the removable singularities theorem have 
a holomorphic extension onto $\C^2$ given by $\frac{\hat u(\zeta_1, \zeta_2)}{\zeta_1 +
i\zeta_2}$. (See e.g. Theorem 7.3.3 in \cite{SG}).

We proceed to show \eqref{desired expansion of f}. By the fact that $\hat u(\zeta_1,\zeta_2)$ is entire on $\C^2$ it has a convergent power series expansion in powers of $\zeta_1$ and $\zeta_2$ which we can write as
\[\hat u(\zeta_1, \zeta_2) = \sum\limits_{j, k = 0}^\infty c_{j,k} (\zeta_1 + i\zeta_2)^j(\zeta_1 -i\zeta_2)^k.\]
Setting $\eta_1 = \eta_2 = 0$ so that $\zeta_j = \xi_j$ we have that, by denoting $\xi = \xi_1 + i\xi_2$, 
\[\hat u(\xi_1, \xi_2) = \sum\limits_{j, k = 0}^\infty c_{j,k} \xi^j\bar\xi^k.\]
We observe that \eqref{desired expansion of f} is equivalent to the fact that the above expansion has $c_{0,n} = 0$ for all $n$. To this end, \eqref{condition at zero} reads \[\xi \hat f(\xi_1,\xi_2) = \sum\limits_{j, k = 0}^\infty c_{j,k} \xi^j\bar\xi^k\]
for $\hat f$ smooth near the origin which immediately gives $c_{0,0} = 0$. Observe that since $\hat f$ is smooth, one can also divide by $\xi$ to get the smooth function
\[\hat f(\xi_1,\xi_2) = \frac{1}{\xi}\sum\limits_{j, k = 0}^\infty c_{j,k} \xi^{j}\bar\xi^k.\]
The right hand side is a-priori defined only on the punctured plane but extends smoothly to $\C$ due to the smoothness of $\hat f$. We will
now hit both sides with the operator $\xi (\frac{\partial}{\partial
\bar\xi})^n$ and take $\xi\to 0$ to get $0 = c_{0,n}$ for all $n\geq 1$ and
\eqref{desired expansion of f} is established.

It remains to apply Paley-Wiener to conclude the super-exponetially decay of $f$. To do so, one needs to check
\[\sup\limits_{|\eta| \leq \gamma} \| \hat f(\cdot + i \eta)\|_{L^2}^2 <\infty,\ \ \ \ \forall \gamma \geq 0.\]
On the strip $|\eta|\leq \gamma$ the vanishing set of $\zeta_1 + i\zeta_2$ is contained in a compact rectangle. On this rectangle $\hat f(\zeta_1,\zeta_2)$ is of course bounded. Outside of this rectangle the estimate comes from the fact that $\hat f = \frac{\hat u}{\zeta_1 + i\zeta_2}$  and the estimate \eqref{FT of f estimate}.
\qed

\begin{lemma}
\label{solving dbar with rapid decay}
Let $f$ be a smooth function on $M_0$ satisfying \[\bar\partial f  \in e^{-\gamma /x} H^m_b(M_0; ^bT^*_{0,1}M_0)\ \ \forall\gamma>0.\] Suppose for all $J\in \R$, 
\begin{eqnarray}
\label{ortho to all holomorphic form}
\int_{M_0} \langle \bar\partial f, \eta\rangle_{g_0} {\rm dvol}_{g_0}= 0, \forall \eta\in x^{-J} L^2(M_0; T^*M_0),\ \bar\partial^* \eta = 0
\end{eqnarray}
then there exists a holomorphic function $\Psi$ such that $\Psi - f \in e^{-\gamma /x} H^m_b(M_0)$ for all $\gamma>0$.
\end{lemma}
\noindent{\bf Proof}
We first find a solution $U$ to the equation 
\begin{eqnarray}
\label{solve with rapid decay}
\bar\partial U = -\bar\partial f,\ \ \ \ U \in x^J H^m_b(M_0)\ \ \forall J\in\R.
\end{eqnarray}
Once such a solution is constructed, the proof is complete by evoking Lemma \ref{rapid decay means exponential decay} to conclude that $U$ belongs to $e^{-\gamma/x}H^m_b(M_0)$ for all $\gamma >0$. To this end, as $\bar\partial f$ decays super-exponentially, we may consider $\bar\partial f$ to be a $x^J H^m_b$ section of $^bT^*M_0$ for all $J\in \R$ and observe that by \eqref{ortho to all holomorphic bform} it is an element of $ R_{J,m}(\bar\partial)$ if and only if 
\[\int_{M_0} \langle \bar\partial f,\eta\rangle_{g_b} {\rm dvol}_{g_b} = 0, \forall \eta \in N_{-J, m}(^b\bar\pl^* ).\]
Using the relation \eqref{relate b-bundle to original bundle} combined with the fact that \[\langle \bar\partial f, \eta\rangle_{g_0} {\rm dvol}_{g_0}=\langle \bar\partial f, \eta\rangle_{g_b} {\rm dvol}_{g_b}\] one sees that the condition given by \eqref{ortho to all holomorphic form} does indeed imply the above orthogonality condition for all $J\in \R$. 

Therefore for each $ J \in \R\backslash\zz$ one can find a solution $U_{ J} \in x^{ J} H^{m}_b(M_0)$ 
 solving $\bar\pl U_{J} = -\bar\pl f$. Since the difference of two such solutions are holomorphic, uniqueness follows for $J \notin\zz$ large by standard arguments for holomorphic functions. Therefore, as $x^J H^m_b \subset x^{J'}H^m_b$ for $J \geq J'$, there exists a unique solution $U\in \bigcap\limits_{J \in \R} x^J L^2$ of $\bar\pl U = -\bar\pl f$. This shows that \eqref{solve with rapid decay} has a unique solution and the proof is complete.

\qed

\begin{remark}
Note that in neither the statement nor the proof of this Lemma is it required for $f$ to have a polyhomogenous expansion.
\end{remark}
By Lemma \ref{solving dbar with rapid decay} we see that to prove Proposition \ref{exponentially decaying conjugating factor} it suffices to show that $e^{i(\alpha_1-\alpha_2)}$ satisfies the orthogonal condition \eqref{ortho to all holomorphic form}. To this end we first derive the following identity:

\begin{lemma}
\label{integral id 1}
Let $u_j \in e^{\gamma/x} H^1(M_0)$ be solutions to $L_{X_j, V_j} u_j =0$ then the integral identity holds
\[ 0 = \int_{M_0} \bar u_1 (A_1- A_2)\wedge \partial u_2 +\bar u_1 (\bar A_1 - \bar A_2)\wedge \bar\partial u_2 + \bar u_1(Q_1 - Q_2)u_2 \]
\end{lemma}
\begin{proof} By Lemma \ref{boundary pairing} and the fact that $S_{X_1,
V_1}(\lambda) = S_{X_2, V_2}(\lambda)$ we have that the above identity
holds for $u_j \in x^{-\tau} H^1(M_0)$ in the range of
$P_{X_j,V_j}(\lambda)$. We first fix $u_2$ and 
use the $e^{\gamma/x}L^2$ density result of Proposition
\ref{density} to take the limit in $u_1$ to conclude that the above
identity holds for $u_2$ in the range of $P_{X_2,V_2}(\lambda)$ and
solutions $u_1 \in e^{\gamma/x} H^1(M_0)$. Since now $u_1$ has
$e^{\gamma/x} H^1(M_0)$ regularity we can take the limit in $u_2$ in the
$e^{\gamma/x}L^2$ topology to obtain the result.
\end{proof}
\noindent{\bf Proof of Proposition \ref{exponentially decaying conjugating factor}}:\\
We begin by choosing $\Phi$ a holomorphic morse function which grows linearly at each end; this function exists by Lemma \ref{poles at ends}. Let ${\rm Crit}(\Phi):=\{p_0,\dots, p_n\}$ be the critical points of $\Phi$ and, for some $J\in \R$, let $b\in x^J L^2(M_0)$ be an antiholomorphic 1-form on $M_0$ which vanishes to third order on points in ${\rm Crit}(\Phi)$. Consider the ansatz $u_0 = he^{\bar\Phi/h} e^{-i\bar\alpha_2} \frac{b}{\bar\partial \bar\Phi}$. By writing
\[L_{X,V} = e^{-i\bar\alpha}\bar\pl^* |e^{-i\alpha}|^2\bar\pl e^{i\alpha} +Q\] for $Q = |X|^2 + V +dX \in e^{-\gamma/x}L^\infty$ for all $\gamma\in \R$, we see that $u_0$ solves
\[(L_{X_2, V_2} - \lambda^2) u_0 = he^{\bar\Phi/h} e^{-i\bar\alpha_2} f\]
with $f \in x^{J'} L^2(M_0)$ for some $J'\in \R$. We now apply Proposition \ref{semiclassical solvability} to obtain a solution to $(L_{X_2,V_2} - \lambda^2) u_2 = 0$ of the form
\[u_2 = u_0 + e^{\bar\Phi/h} e^{-i\bar\alpha_2} r_2 \]
with $r_2$ satisfying the estimate $\|e^{\varphi_0/\eps} r_2 \| + h\|e^{\varphi_0/\eps} dr_2\| \leq h\sqrt{h} C\|x^{-J'}f\|$ where $\varphi_0$ is as required in definition \eqref{eq_phieps}. 
For the solution $(L_{X_1, V_1} - \lambda^2) u_1 = 0$ we use the ansatz
\[(L_{X_1, V_1} - \lambda^2) e^{\Phi/h} e^{-i\alpha_1} = e^{\Phi/h}e^{-i\alpha_1} f\]
with $f\in x^J L^2$ for some $J\in \R$. Proposition \ref{semiclassical solvability} again applies to obtain a solution to $(L_{X_1,V_1} - \lambda^2) u_1 = 0$ of the form
\[u_1 = e^{\Phi/h} e^{-i\alpha_1} + e^{\Phi/h} e^{-i\alpha_1} r_1.\]
with $r_1$ satisfying the estimate $\|e^{\varphi_0/\eps} r_1 \| + h\|e^{\varphi_0/\eps} dr_1\| \leq \sqrt{h} C\|x^{-J}f\|$ where $\varphi_0$ is as required in definition \eqref{eq_phieps}. 
 
 We now substitute these solutions into the identity in Lemma \ref{integral id 1} to obtain, after taking $h\to 0$, 
 \[ 0 = \int_{M_0} \langle \bar\partial e^{i(\alpha_1-\alpha_2)}, b\rangle {\rm dvol}_{g_0}\]
 for all anti-holomoprhic 1-forms $b$ vanishing to third order at ${\rm Crit}(\Phi) = \{p_0,\dots, p_n\}$ which are in the space $x^J L^2(M_0)$ for some $J\in \R$.
 
We do not have the orthogonal condition \eqref{ortho to all holomorphic form}
for all anti-holomorphic 1-forms yet because of the restricted vanishing
condition. We will get rid of the vanishing condition one point at a time
starting with $p_0$. To this end, we use Lemma \ref{holofcts} and Corollary
\ref{morse holomorphic dense} to construct a holomorphic Morse function
$\tilde\Phi$ for which $p_0 \notin {\rm Crit}(\tilde \Phi) := \{\tilde
p_0,\dots, \tilde p_m\}$. Repeating the above argument for $\tilde\Phi$ we have
that 
\[ 0 = \int_{M_0} \langle \bar\partial e^{i(\alpha_1-\alpha_2)}, \tilde
b\rangle {\rm dvol}_{g_0}\]
for all anti-holomoprhic 1-forms $\tilde b$ vanishing to third order at 
\[
\]

Let $b\in x^J L^2(M_0)$ for some $J\in \R$ be an antiholomorphic 1-form which vanishes at $\{p_1,\dots, p_n\}$. By Lemma \ref{amplitude}, it can be written as the sum $(b -\tilde b)+\tilde b$ where $(b -\tilde b) \in x^{\tilde J} L^2(M_0)$ vanishes to third order at $\{p_0,\dots, p_n\}$ and $\tilde b \in x^{\tilde J} L^2(M_0)$ vanishes to third order at $\{\tilde p_0,\dots, \tilde p_m\}$. Linearity then implies that 
\[ 0 = \int_{M_0} \langle \bar\partial e^{i(\alpha_1-\alpha_2)}, b\rangle {\rm dvol}_{g_0}\]
for all antiholomorphic 1-forms $b \in x^J L^2(M_0)$ vanishing to third order at $\{p_1,\dots, p_n\}$. Proceeding as such for the points $p_1, \dots, p_n$ we can remove all the vanishing conditions placed upon $b$ and the orthogonality condition \eqref{ortho to all holomorphic form} holds for all antiholomorphic 1-forms. The existence of $\Psi$ satisfying the asymptotic condition \eqref{boundary condition of Psi} is thus proven by evoking Lemma \ref{solving dbar with rapid decay} and observing that $U\in e^{-\gamma/x} H^m_b(M_0)$ for all $\gamma>0$ iff $U\in e^{-\gamma/x} H^m(M_0)$ for all $\gamma>0$.

To see that the holomorphic function $\Psi$ is non-vanishing, we interchange the indices to deduce that there exists holomorphic functions $\Psi_{1,2}$ and $\Psi_{2,1}$ on $M_0$ which satisfies condition \eqref{boundary condition of Psi} for $e^{i(\alpha_1-\alpha_2)}$ and $e^{i(\alpha_2-\alpha_1)}$ respectively. Considering the product $\Psi_{1,2} \Psi_{2,1}$ and using condition \eqref{boundary condition of Psi} we see that the product is actually the constant function $1$. \qed

If $\alpha_j$ are the functions constructed in Proposition \ref{constructing alpha}, it is convenient make the definition $F_{A_1} := e^{i\alpha_1}$ and $F_{A_2} := \Psi e^{i\alpha_2}$ where $\Psi$ is the holomorphic function constructed in Proposition \ref{exponentially decaying conjugating factor}. Following the construction of $\Psi$ and using the fact that $H^{3+\eps_0}(M_0)\subset W^{2,\infty}(M_0)$ one has the following useful expression for $F_{A_2}$
\begin{lemma}
\label{FA2 expression}
We have that $F_{A_2} = e^{i\alpha_1} ( 1+ t)$ with $t$ bounded uniformly away from $-1$ and belongs to the space $e^{-\gamma/x} W^{2,\infty}(M_0)$ for all $\gamma>0$.
\end{lemma}

\end{section}

\begin{section}{Construction of CGO Solutions} \label{sec_cgo}
We construct special solutions to $(L_{X,V} - \lambda^2) u = 0$. To this end it is convenient to write the differential operator in terms $\bar\partial$ and its adjoint. Namely, if $\alpha$ is a function such that $\bar\partial \alpha = A$ then one can write \[L_{X,V} = e^{-i\bar\alpha}\bar\pl^* |e^{-i\alpha}|^2\bar\pl e^{i\alpha} +Q\] for $Q = |X|^2 + V +dX \in e^{-\gamma/x}L^\infty$ for all $\gamma\in \R$. We would like to study the existence of such functions $\alpha$ with suitable behaviour near the ends.

If $M_0$ is a surface with $N$ Euclidean ends, consider its $N$ point compactification $M$ by adding the points $\{e_1,\dots, e_N\}$ at the ends. Around each $e_j$ introduce holomorphic coordinate $z$ and write $z = xe^{i\theta}$.


\begin{subsection}{Constructing CGO of Type I}
We are now in a position to construct a family CGO which will be useful for recovering interior information. Let $b$ be a section of $T^*_{0,1} M$ belonging to $Ker(\bar\partial^*)$ with poles contained in the set $\{e_1,\dots, e_N\}$ and $\Phi = \phi + i\psi$ be a morse holomorphic function on $M_0$ with poles of the form $\frac{C_j}{z}$ near $e_j$,  $C_j \neq 0$. Let $F_A$ be a smooth function on $M_0$ which satisfies $F_A - e^{i\alpha} \in e^{-\gamma/x} W^{2,\infty}(M_0)$ for all $\gamma >0$. If $\chi\in C^\infty_0(M_0)$ is a cutoff function which is $1$ near all the critical points of $\Phi$, consider the ansatz
\begin{eqnarray}
\label{u0}
u_0 = e^{\Phi/h}  F_A^{-1} \bar\partial_J^{-1} (e^{-2i\psi/h}\chi |F_A|^2b) - h (1-\chi)e^{\bar\Phi/h} \bar F_A \frac{b}{2i\bar\pl\psi}
\end{eqnarray}
for $2>J>1$. Here we use the notation $(1-\chi)\frac{b}{\bar\partial \psi}$ to denote the function satisfying $\bar\partial\psi (1-\chi)\frac{b}{\bar\partial \psi} = (1-\chi)b$. It is well-defined since $\psi$ has no critical points on the support of $(1-\chi)$.

By writing $L_{X,V} =  \bar F_A\bar\pl^* |F_A|^{-2} \bar\pl F_A + Q$ direct computation yields that 
\[(L_{X,V} -\lambda^2) u_0 = -\lambda^2 u_0 + he^{\bar\Phi/h}  e^{-\gamma/x} L^2(M_0)\]
for all $\gamma >0$. Consequently we can use the estimates we established in Lemma \ref{estimate for conjugate dbar} and the expression for $u_0$ to obtain
\begin{eqnarray}\label{(L-lambda)u0} (L_{X,V} -\lambda^2) u_0  = e^{\Phi/h} F_A^{-1} O_{x^{-J} L^2}(h^{\frac{1}{2} + \epsilon}) +he^{\bar \Phi/h} \bar F_A x^{-J'}L^2 \end{eqnarray}
for some $J' >0$ and $2>J>1$. By Proposition \ref{semiclassical solvability} we can solve for the remainder $r$ so that $(L_{X,V} -\lambda^2)(u_0 + e^{\varphi/h} r) = 0$ with $r$ satisfying the estimate $\|e^{-\gamma_0/x} r\| + h\|e^{-\gamma_0/x} dr\| \leq C h^{1+\epsilon}$ for some $\gamma_0 >0$. We summarize this discussion in the following Proposition.
\begin{proposition}
\label{CGO type I}
There exists solutions to $(L_{X,V} -\lambda^2) u =0$ of the form $u = u_0 + e^{\phi/h} r$ where $u_0$ is given by \eqref{u0} and $r$ satisfies the estimate \[\|e^{-\gamma_0/x} r\| + h\|e^{-\gamma_0/x} dr\| \leq C h^{1+\epsilon}\] for some $\gamma_0 >0$.
\end{proposition}

\end{subsection}
\begin{subsection}{Constructing CGO of Type II}

Let $\Phi= \phi + i\psi$ be a holomorphic Morse function which has critical points $\{p_0,\dots, p_n\}$ and expansion $\Phi = \frac{C_j}{z}$ for $C_j\neq 0$ near the ends $e_j$ for $j=1,\dots ,N$. Let $a$ be a holomorphic function in $x^{-J} L^2(M_0)$ for some $J \in \R_+\backslash\zz$ and which vanishes at $\{p_1,\dots, p_n\}$ but does not vanish at $p_0$. We see then that \[(L_{X,V} - \lambda^2) e^{\Phi/h} e^{-i\alpha} a = (Q - \lambda^2)e^{\Phi/h} e^{-i\alpha} a\] and this motivates us to seek $r_1$ solving 
\[(L_{X,V} - \lambda^2) e^{\Phi/h}r_1 = - (Q- \lambda^2) e^{\Phi/h} e^{-i\alpha} a + e^{\Phi/h}e^{-i\alpha} O_{x^{-J} L^2}(h|{\log}h|).\]
To this end,
let $G$ be the operator of Lemma \ref{rightinv}, mapping continuously $x^{-J+1}L^2(M_0)$ 
to $x^{-J-1}L^2(M_0)$. Then clearly ${\partial}^*\partial G=Id$  when acting on $x^{-J+1}L^2$.

First, we will search for $r_1$ satisfying
\begin{equation}
\label{dequation}
e^{-2i\psi/h}|e^{i\alpha}|^2\partial e^{i\bar\alpha} e^{2i\psi/h} r_1 = -\pl G (a(Q-\la^2)) + \omega + O_{x^{-J}H^1}(h |\log h|)
\end{equation}
 with $\omega\in x^{-J}L^2(M_0)$ a holomorphic 1-form on $M_0$  
and $\|r_1\|_{x^{-J}L^2} = O(h )$. 
Indeed, using the fact that $\Phi$ is holomorphic we have
\[e^{-\Phi/h} (L_{X,V} - \lambda^2) e^{\Phi/h}= e^{-i\alpha} \pl^* e^{-2i\psi/h} |e^{i\alpha}|^2\pl e^{i\bar\alpha} e^{2i\psi/h} + \bar Q -\lambda^2\] 
for some smooth superexponentially decaying function $Q$ and applying $e^{-i\alpha}\pl^*$ to \eqref{dequation}, this gives
\[e^{-\Phi/h}(L_{X,V} -\lambda^2)e^{\Phi/h}r_1=-ae^{-i\alpha}(Q-\la^2)+e^{-i\alpha}O_{x^{-J}L^2}(h |\log h|).\]
Writing $-\pl G(a(Q-\la^2))=:c(z)dz$ in local complex coordinates, $c(z)$ is $C^{2,\gamma}$ by elliptic regularity and 
we have $2i\pl_{\bar{z}}c(z)=a(Q-\la^2)$, therefore $\pl_z\pl_{\bar{z}}c(p')=\pl^2_{\bar{z}}c(p')=0$ at each critical point $p'\not=p_0$ by construction 
of the function $a$.  Therefore, we deduce that at each critical point $p'\neq p_0$, $c(z)$ 
has Taylor series expansion $\sum_{j = 0}^2 c_j z^j + O(|z|^{2+\gamma})$. That is, all the lower order terms of the Taylor expansion of $c(z)$
 
around $p'$ are polynomials of $z$ only. By Lemma \ref{control the zero}, and possibly by taking $J$ larger, there exists a holomorphic function $f\in x^{-J}L^2$ 
such that $\omega:=\pl f$ 
has Taylor expansion equal to that of $\pl G(a(Q-\la^2))$ at all critical points $p'\not=p_0$ of $\Phi$. We deduce that, if
 $\beta:=-\pl G(a(Q-\la^2))+\omega=\beta(z)dz$, we have
\begin{equation}\label{decayofb}
\begin{gathered}
|\pl_{\bar{z}}^m\pl^{\ell}_z \beta(z)|=O(|z|^{2+\gamma-\ell-m}) , \quad \textrm{ for } \ell+m\leq 2 , \textrm{ at critical points }p'\not=p_0\\
|\beta(z)|=O(|z|) , \qquad \qquad \qquad \qquad \textrm{ if }p'=p_0.  
 \end{gathered}\end{equation}
 Now, we let $\chi_1\in C_0^\infty(M_0)$ be a cutoff function supported in a small neighbourhood $U_{p_0}$ of the critical point $p_0$ and identically $1$ near 
 $p_0$, and 
$\chi\in C_0^\infty(M_0)$ is defined similarly with $\chi =1$ on the support of $\chi_1$.
We will construct $r_1$ to be a sum $r_1=r_{11} +h r_{12}$ where  $r_{11}$ is a compactly supported 
approximate solution of \eqref{dequation}  near the critical point $p_0$ of $\Phi$ and $r_{12}$ is correction term supported 
away from $p_0$.
We define locally in complex coordinates centered at $p_0$ and containing the support of $\chi$
\begin{equation}\label{defr11}
r_{11}:=\chi e^{-2i\psi/h} e^{-i\bar\alpha}R(|e^{i\alpha}|^{-2}|e^{2i\psi/h}\chi_1\beta)
\end{equation}
where $Rf(z) := -(2\pi i)^{-1}\int_{\R^2} \frac{1}{\bar{z}-\bar{\xi}}f d\bar{\xi}\wedge d\xi$ for $f\in L^\infty$ compactly supported 
is the classical Cauchy operator inverting locally $\pl_z$ ($r_{11}$ is extended by $0$ outside the neighbourhood of $p$).
The function $r_{11}$ is in $C^{3,\gamma}(M_0)$ and we have 
\begin{equation}\begin{gathered}\label{r11}
e^{-2i\psi/h}|e^{i\alpha}|^2\pl(e^{2i\psi/h} e^{i\bar\alpha}r_{11}) = \chi_1(-\pl G(a(Q-\la^2)) + \omega) + \eta\\
\textrm{ with }\eta:= e^{-2i\psi/h}e^{-i\bar\alpha}R(|e^{i\alpha}|^{-2}e^{2i\psi/h}\chi_1\beta)\pl\chi.
\end{gathered}
\end{equation}
We then construct $r_{12}$ by observing that  $b$ vanishes to order $2+\gamma$ at critical points of $\Phi$ other than $p_0$ (from \eqref{decayofb}), and 
$\pl \chi=0$ in a neighbourhood of any critical point of $\psi$, so we can find $r_{12}$ satisfying
\begin{equation}\label{defr12}
2i e^{i\alpha}r_{12}\pl\psi = (1-\chi_1)\beta .
\end{equation} 
This is possible since both $\pl\psi$ and the right hand side are valued in $T^*_{1,0}M_0$ and 
$\pl \psi$ has finitely many isolated zeroes on $M_0$: 
$r_{12}$ is then a function which is in $C^{2,\gamma}(M_0\setminus{P})$ where $P:=\{p_1,\dots, p_n\}$ is the set of critical points other than $p_0$,
it extends to a function in $C^{1,\gamma}(M_0)$  and it satisfies in local complex coordinates $z$ at each $p_j$, $j =1,\dots, n$ 
\[ |\pl_{\bar{z}}^j\pl_z^k r_{12}(z)|\leq C|z|^{1+\gamma-j-k} , \quad j+k\leq 2\]
by using also the fact that $\pl \psi$ can be locally be considered as a smooth function with a zero of order $1$ at each $p_j$.
Moreover $\beta\in x^{-J}H^2(M_0)$ thus  $r_1\in x^{-J}H^2(M_0)$ and  we have 
\[ e^{-2i\psi/h}|e^{i\alpha}|^2\pl(e^{i\bar\alpha}e^{2i\psi/h}r_1) =-\pl G(a(Q-\lambda^2))+\omega+ h|e^{i\alpha}|^2\pl e^{i\bar\alpha}r_{12} + \eta.\]

\begin{lemma}\label{fewestimates}
The following estimates hold true 
\[\begin{gathered} 
||\eta||_{H^2(M_0)}=O(|\log h|),\; \|\eta\|_{H^1(M_0)}\leq O(h|\log h|), \;
||x^{J}\pl r_{12}||_{H^1(M_0)}=O(1),\\
 ||x^Je^{i\alpha}r_{1}||_{L^2}=O(h), \; ||x^Je^{i\alpha}(r_1-h\til{r}_{12})||_{L^2}=o(h) 
\end{gathered} \]
where $\til{r}_{12}$ solves $2ie^{i\alpha}\til{r}_{12}\pl\psi = \beta$.
\end{lemma}
\begin{proof} The proof is exactly the same as the proof of Lemma 4.2 in \cite{GT2}, except that one needs to add the weight 
$x^J$ to have bounded integrals.\end{proof}
As a direct consequence, we have 
\begin{cor}\label{corerrorterm}
With $r_1=r_{11}+h r_{12}$, there exists $J > 0$ such that 
\[||e^{i\alpha}e^{-\Phi/h}(L_{X,V}- \la^2)e^{\Phi/h}(a + r_1)||_{x^{-J}L^2(M_0)} = O(h|\log h|). \]
\end{cor}
Now we can apply Proposition \ref{semiclassical solvability} to obtain solutions to $(L_{X,V} -\la^2) u = 0$ of the form
\begin{eqnarray}
\label{type II CGO}
u = e^{\Phi/h}(a + r_1) + e^{\varphi/h}r_2
\end{eqnarray}
with $r_2$ satisfying the estimates
\[\|e^{-\gamma_0/x} r_2\| + h\|e^{-\gamma_0/x} dr_2\| \leq C h^{1+ \frac{1}{2}} |\log h|\]
for some $\gamma_0>0$. 
\end{subsection}
\end{section}

\section{Conjugation Factors and an Integral Identity}\label{sec_conj}

\noindent
We begin by defining the functions $F_{A_1}$ and $F_{A_2}$ as
\begin{align} \label{eq_def_FA}
    F_{A_1} := e^{i \alpha_1} \quad\text{ and } \quad   F_{A_2} := \Psi e^{i \alpha_2},
\end{align}
where $\alpha_j$ are the soultions to $\dbar\alpha_j = A_ j$, given by
Proposition \ref{constructing alpha} and 
where $\Psi$ is the holomorphic function given by Proposition
\ref{exponentially decaying conjugating factor} so that $F_{A_2}$ has the
expression given by Lemma \ref{FA2 expression}

We proceed by first deriving an apropriate system and from this an integral identity.
Consider the equation
\begin{align} \label{eq_LXV}
(L_{X_j,V_j} -\lambda^2)u_j = 0, \quad \text{ on } \quad M_0,
\end{align}
$j=1,2$.
This equation can rewritten by means of the above defintions
in the  form
\[
    \big(2F_{\ov{A}_j}^{-1} \dbar^* F_{\ov{A}_j}  F_{A_j}^{-1} \dbar F_{A_j} -\lambda^2 + Q_j \big) u_j = 0,
\]
where $Q_j := \star dX_j + V_j$.
Using this and
by defining the function and 1-form
\[
    \tilde u_j := F_{A_j} u_j,\quad
    \tilde \omega := F_{\ov{A}_j}  F_{A_j}^{-1}\dbar\tilde u_j = |F_{A_j}|^{-2}\dbar\tilde u_j,
\]
and further setting  
\begin{align*}
D := 
  \begin{pmatrix}
   0  &\bar\partial^*\\
  \bar\partial&  0
 \end{pmatrix} ,
 \;
 \mathcal{A}_j := 
  \begin{pmatrix}
      (Q_j -\lambda^2)|\bar F_{A_j}|^2  &0\\
      0 &- |F_{A_j}|^2
 \end{pmatrix},
 \;
 U_j :=
  \begin{pmatrix}
      \tilde u_j  \\
      \tilde \omega_j 
 \end{pmatrix},
\end{align*}
on sees that equation \eqref{eq_LXV} is then equvivalent to the system
\[
    (D+\mathcal{A}_j) U_j = 0, \quad \text{ on } \quad M_0.
\]

In order to derive the integral identity we define
an exhaustion of $M_0$, given by 
the sets $M_R := \tilde M \cup \tilde E_{1,R} \cup \dots \cup \tilde E_{N,R}$, where
$\tilde M \setminus \cup_j E_j$ and $\tilde E_{j,R} \simeq B(0,R) \setminus B(0,1)$, $R>1$.
Consider the solutions
$U_j$ of $(D+\mathcal{A}_j) U_j = 0$, $j=1,2$. Letting
$W := U_1-U_2$, we have that
\[
    \int_{M_R} \big\langle (D+\mathcal{A}_1) W, U_1 \big\rangle 
    =
    \int_{M_R} \big\langle (\mathcal{A}_2 - \mathcal{A}_1) U_2, U_1 \big\rangle,
\]
where $\langle U_1 , U_2 \rangle := \langle \tilde u_1 , \tilde u_2 \rangle 
+ \langle \tilde \omega_1 , \tilde \omega_2 \rangle$ where inner products on the right side 
are the ones induced by the metric.
By integrating by parts the left hand side and using the fact that $(D -\mathcal{A}_1) U_1=0$,
we have that
\begin{align}\label{eq_bndry}
    \int_{M_R} \big\langle (\mathcal{A}_2 - \mathcal{A}_1) U_2, U_1 \big\rangle 
    =
    \int_{ \p M_R} \iota^* \big((\tilde u_1- \tilde u_2) \star \ov{\tilde \omega_1}
    - \star (\tilde\omega_1 - \tilde \omega_2) \ov{\tilde u_1} \big),   
\end{align}
where $\iota\colon \p M_R \to M_0$ is the inclusion map.
The boundary $\p M_R$ can be decomposed into components that are contained in the ends $E_j$.
The integration set $\p M_R$ can moreover be considered to consist of
the set $\{ z \in \C  \colon |z|=R\}$ in each end.

We now let $u_j \in x^{-\tau}L^2$, $\tau >\demi$ be the scattering solutions
given by means of the Poisson operator, i.e. $u_j = P_j(\lambda) g$, where $g \in C^\infty(\p M_0)$
and consider how
the boundary integral in \eqref{eq_bndry} behaves, when $R \to \infty$.
By Proposition \ref{poisson kernel} we have the following asymptotics
for the scattering solutions,
\[
u_j = P_j(\lambda) g
= c_\lambda r^{-\frac{1}{2}} [e^{i\lambda r} (S_{X_j,V_j}(\la) g)(\theta)) +
i e^{-i\lambda r}g(\theta)] +{\cal R}, 
\]
where $z = re^{i\theta}$ and $|{\cal R}| + |\nabla {\cal R}| \leq Cr^{-\frac{3}{2}}$.

To see that the boundary integral in \eqref{eq_bndry} vanishes in the limit, we firstly note that
\begin{align*} 
(\tilde u_1 - \tilde u_2)\ov{\tilde \omega_1} 
= (F_{A_1}u_1 - F_{A_2}u_2) \ov{\tilde \omega_1} 
&= (F_{A_1}(u_1 - u_2) + (F_{A_1} - F_{A_2})u_2) \ov{\tilde \omega_1}   \nonumber 
\end{align*}
The term containing $F_{A_1} - F_{A_2}$ decays super exponentially thanks to Proposition
\ref{exponentially decaying conjugating factor} and will therefore not contribute to the integral in 
\eqref{eq_bndry} in the limit. We have moreover that
\begin{align*}
F_{A_1}(u_1 - u_2)\ov{\tilde \omega_1} 
&= 
F_{A_1}(u_1 - u_2)|F_{A_1}|^{-2}(\ov{F_{A_1}} \p \ov{u_1} + \ov{u_1} \p \ov{F_{A_1}}) \\
&=
(u_1 - u_2) \p \ov{u_1}  + O(e^{-r}).
\end{align*}
The scattering matricies are equal for the potentials, i.e. $S_{X_1,V_1}(\la) g = S_{X_2,V,2}(\la) g$.
This together with the above asymptotics for $u_j$
imply that the above expression is $O(r^{-3/2})$.
It follows the last term in the integral in \eqref{eq_bndry} vanishes in the limit $R \to \infty$.

To handle the $\tilde \omega_1 - \tilde \omega_2$ term in
\eqref{eq_bndry}, we argue similarly. First note that 
the derivative has the expansion
\[
 \dbar u_j  
= c_\lambda r^{-\frac{1}{2}} [ i\lambda e^{i\lambda r} S_{X_j,V_j}(\la) g(\theta)) +
\lambda e^{-i\lambda r}g(\theta)] + O(r^{-\frac{3}{2}}).
\]
Secondly we have that
\[
\tilde \omega_j 
= |F_{A_j}|^{-2}\dbar F_{A_j} u_j - |F_{A_j}|^{-2}F_{A_j} \dbar u_j
= - |F_{A_j}|^{-2}F_{A_j}  \dbar u_j + O(e^{-r}),
\]
It follows that
\begin{align*}
    (\tilde \omega_1 - \tilde \omega_2) \ov{\tilde u_1}
    &= 
    \Big(  \frac{F_{A_2}}{|F_{A_2}|^2} \dbar u_2
    - \frac{F_{A_1}}{|F_{A_1}|^2} \dbar u_1    
    \Big) \ov{F}_{A_1} \ov{u_1} + O(e^{-r}) \\
    &= 
    \big(  (\ov{F}_{A_1}  - \ov{F}_{A_2}) \dbar u_2
    + \ov{F}_{A_2} (\dbar u_2 - \dbar u_1)    
    \big) \frac{F_{A_2}}{|F_{A_2}|^2} \ov{u_1} + O(e^{-r}), 
\end{align*}
the first term decays super exponentially
by  Proposition \ref{exponentially decaying conjugating factor}.
In the second term the $F_{A_2}$:s  cancel, and we can thus use the asymptotics of $\dbar u_j$ 
together with the fact that $S_{X_1,V_1}(\la) g = S_{X_2,V,2}(\la) g$  to see that
it is $O(r^{-3/2})$. 
This implies that the integral containing the $\tilde \omega_1 - \tilde \omega_2$ term 
in \eqref{eq_bndry}, will vanish, when $R \to \infty$.

By taking the limit $R \to \infty$ in \eqref{eq_bndry} we obtain hence that
\begin{align} \label{eq_inteq1}
    \int_{M_0} \big\langle (\mathcal{A}_2 - \mathcal{A}_1) U_2, U_1 \big\rangle 
    = 0,
\end{align}
when $U_1$ and $U_2$  are made up of scattering solutions. 
As a consequence of the density result of Proposition \ref{density} we can extend this to
all exponentially growing solutions of \eqref{eq_LXV}, which is the content of the  following Lemma.

\begin{lemma} \label{intid}
Assume that $S_{X_1,V_1}(\la) g = S_{X_2,V,2}(\la) g$.
Let $v,w \in e^{\gamma/x}H^1(M_0)$ for some $\gamma>0$, be solutions of
$(L_{X_1,V_1}-\lambda^2)v=0$ and $(L_{X_2,V_2}-\lambda^2)w=0$, on $M_0$, then
\begin{align*}
    \int_{M_0} \big\langle (|F_{A_1}|^{-2}- |F_{A_2}|^{-2} )\dbar \tilde v , \dbar \tilde w \big\rangle 
    + \frac{1}{2} \big\langle (Q_1|F_{A_1}|^{2}-Q_2|F_{A_2}|^{2} ) \tilde v , \tilde w \big\rangle 
    = 0,
\end{align*}
\end{lemma}

\begin{proof}
Proposition \ref{density} implies that we can pick two sequences of scattering solutions 
$(v_k)$ and $(w_k)$, s.t.
$v_k \to v$, and $w_k \to w$ in the $e^{\gamma'/x}L^2$-norm. It follows that
$\tilde v_k \to  \tilde v$, and $ \tilde w_k \to \tilde  w$ in the $e^{\gamma'/x}L^2$-norm and
moreover that $ \dbar \tilde v_k \to  \dbar \tilde v$, and 
$ \dbar \tilde w_k \to \dbar \tilde  w$ in the $e^{\gamma'/x}H^{-1}$-norm. 

Use the abbreviations $\mu_F := |F_{A_1}|^{-2}- |F_{A_2}|^{-2} \in e^{-\gamma/x} W^{1,\infty}(M_0)$
and $\mu_Q := \frac{1}{2}(Q_1|F_{A_1}|^{2}-Q_2|F_{A_2}|^{2}) \in e^{-\gamma/x}W^{1,\infty}(M_0)$.
Now suppose first that $u$ is a scattering solution to $(L_{X_2,V_2} -\lambda^2)u = 0$.
Since the claim holds for scattering solutions, because of \eqref{eq_inteq1} we have that
\begin{small}
\begin{align}\label{eq_MFQ}
\int_{M_0} \big\langle \mu_F \dbar \tilde v , \dbar \tilde u \big\rangle 
+ \big\langle M_Q \tilde v , \tilde u \big\rangle 
&=
\int_{M_0} \big\langle \mu_F \dbar (\tilde v - \tilde v_k), \dbar \tilde u \big\rangle 
+ \big\langle \mu_Q (\tilde v - \tilde v_k), \tilde u \big\rangle 
\end{align}
\end{small}
We can estimate the term  involving $\mu_Q$ by
\begin{align} \label{eq_MQ}
\int_{M_0} \big\langle \mu_Q (\tilde v - \tilde v_k) , \tilde u \big\rangle 
&\leq
\|e^{\gamma'/x}\mu_Q\|_{L^\infty} \|e^{-\gamma'/x} (\tilde v - \tilde v_k)\|_{L^2} \|\tilde u\|_{L^2} \to 0,
\end{align}
as $k \to \infty$.
To handle the other term in \eqref{eq_MFQ} we write
\begin{align*}
    \big\langle \mu_F \dbar (\tilde v - \tilde v_k) , \dbar \tilde u \big\rangle &=
    (|F_{A_1}|^{-2}- |F_{A_2}|^{-2})F_{A_1}\ov{F}_{A_2} \\
    &\quad(iA_1(v-v_k) - \dbar(v-v_k))(-i\ov{A}_2 \bar u-\p \bar u).
\end{align*}
Furthermore we have that
\begin{small}
\begin{align*}
    \mu :=
    (|F_{A_1}|^{-2}- |F_{A_2}|^{-2})F_{A_1}\ov{F}_{A_2}
    &=
    \big((F_{A_2}- F_{A_1})\ov{F}_{A_2} + F_{A_1}(\ov{F}_{A_2} - \ov{F}_{A_1})\big)/F_{A_2}\ov{F}_{A_1}.
\end{align*}
\end{small}
Proposition \ref{exponentially decaying conjugating factor}
implies that $\mu$ is super exponentially decaying.
To estimate the first term on the right hand side of \eqref{eq_MFQ} we write
\begin{align*}
\int_{M_0} \big\langle \mu_F \dbar (\tilde v - \tilde v_k), \dbar \tilde u \big\rangle 
&=
\int_{M_0} \mu(iA_1(v-v_k) - \dbar(v-v_k))(-i\ov{A}_2 \bar u- \p \bar u) \\
&\leq 
\|A_1(v-v_k) \|_{L^2}  \|A_2 \bar u\|_{L^2} \\
&\quad+\| x^{-\alpha}A_1(v-v_k)\|_{L^2}\| x^{\alpha} \p  \bar u\|_{L^2} \\
&\quad+\| \mu \dbar( v -  v_k)\|_{H^{-1}} \|A_2  \bar u\|_{H^1} \\
&\quad+\| x^{-\alpha}\mu \dbar( v -  v_k)\|_{H^{-1}} \| x^{\alpha} \p  \bar u\|_{H^1}  \\
&\to 0,
\end{align*}
as $k \to \infty$, where we used the fact that $u \in x^{-\alpha}H^2(M_0)$, $\alpha>0$ 
by elliptic regularity.
It follows 
\begin{align} \label{eq_Miid}
\int_{M_0} \big\langle \mu_F \dbar \tilde v , \dbar \tilde u \big\rangle 
+ \big\langle \mu_Q \tilde v , \tilde u \big\rangle = 0.
\end{align}
We can now repeat this form of argument to obtain the above equation with the scattering solution $u$
replaced by $w$, thus proving the claim. 
%
\end{proof}

\section{Gauge equivalence} \label{sec_gauge}

The goal of this section is to prove the gauge equivalence statement of Theorem \ref{thm1}: 
\begin{proposition}
\label{same gauge}
If $S_{X_1,V_1}(\la) = S_{X_2,V,2}(\la)$  for a fixed $\lambda \in \R\backslash\{0\}$ then there exists a unitary function $\Theta$ such that $X_1 - X_2 = d\Theta/\Theta$.
\end{proposition}

Let $\Phi = \varphi + i \psi$ be a Morse holomorphic function given by Lemma \ref{holofcts} and Corollary \ref{morse holomorphic dense}, 
and let $\{ p_0, \dots ,p_n \}$ be its critical points. Let $b$
be an antiholomorphic 1-form chosen so that 
 \begin{eqnarray}\label{condition for b}
 b(p_1)= \dots= b(p_n)=0,\ \ \  b(p_0)\neq 0.\end{eqnarray} 
 Such antiholomrphic 1-forms are given by Lemma \ref{amplitude}. Proposition \ref{CGO type I} gives $u_j$, $j=1,2$ solving 
$(L_{X_j,V_j}-\lambda^2)u_j = 0$,  and are of the form
\[
    u_1 := u_{0,1} + e^{ \varphi/h}r_1,\ \ \ u_2 := u_{0,2} + e^{ -\varphi/h}r_2
\]
where $u_{0,j}$ is given by \eqref{u0} where $u_{0,1}$ is constructed with phase $\Phi$ while $u_{0,2}$ is constructed with phase $-\Phi$.

Let  $\tilde u_j := F_{A_j} u_j$ and direct computation shows that
{\Small\begin{align} \label{eq_dbaruj}
   \dbar \tilde u_j= \dbar(F_{A_j} u_j) = e^{(-1)^{j-1}\bar\Phi/h} |F_{A_j}|^2 b
    + he^{(-1)^{j-1}\bar\Phi/h} |F_{A_j}|^2 R_{0,j} + \dbar(F_{A_j}e^{(-1)^{j-1}\varphi/h} r) 
\end{align}}
where $R_{0,j} \in e^{\gamma_0/x} W^{1,\infty}(M_0) $ for some $\gamma_0 >0$.

Plugging in the expression \eqref{eq_dbaruj} into the identity given by Lemma \ref{intid} we obtain
{\small\begin{eqnarray}\label{recovering FA terms}o(h) &=& \int_{M_0}( |F_{A_1}|^2 - |F_{A_2}|^2) \langle e^{\bar\Phi/h} (b + h R_{0,1}),  e^{-\bar\Phi/h}(b+ hR_{0,2}) \rangle \\\nonumber
&+& \int_{M_0}(|F_{A_1}|^2 - |F_{A_2}|^2) \langle e^{\bar\Phi/h} (b + h R_{0,1}), |F_{A_2}|^{-2}\bar\partial( F_{A_2} e^{-\varphi/h} r_2)\rangle\\\nonumber
&+&\int_{M_0}(|F_{A_1}|^2 - |F_{A_2}|^2)  \langle |F_{A_1}|^{-2}\bar\partial(F_{A_1} e^{\varphi/h} r_1), e^{-\bar\Phi/h}(b+ hR_{0,2})\rangle\\\nonumber
&+& \int_{M_0} (|F_{A_1}|^2 - |F_{A_2}|^2)  \langle  |F_{A_1}|^{-2}\bar\partial(F_{A_1} e^{\varphi/h} r_1),|F_{A_2}|^{-2}\bar\partial( F_{A_2} e^{-\varphi/h} r_2)\rangle.\end{eqnarray}
Note that Lemma \ref{FA2 expression} ensures $(|F_{A_1}|^2 - |F_{A_2}|^2) \in e^{-\gamma/x}W^{1,\infty}(M_0)$ for all $\gamma >0$.

We need to show that everything on the right-side of \eqref{recovering FA terms} is $o(h)$ except the principal term $\int_{M_0} (|F_{A_1}|^2 - |F_{A_2}|^2)   e^{-2i\psi/h} |b|^2$. This comes by direct computation using stationary phase for terms not involving $\bar\partial  (F_{A_j} e^{\pm\varphi/h} r_j)$. For terms which has the same form as \[ \int_{M_0} (|F_{A_1}|^2 - |F_{A_2}|^2)  \langle e^{\bar\Phi/h} b, |F_{A_2}|^{-2}\bar\partial( F_{A_2} e^{-\varphi/h} r_2)\rangle\] we can take advantage of the super-exponential decay of $(|F_{A_1}|^2 - |F_{A_2}|^2)$ and integrate-by-parts
\begin{eqnarray*} &&\int_{M_0}(|F_{A_1}|^2 - |F_{A_2}|^2)  \langle e^{\bar\Phi/h} b, |F_{A_2}|^{-2}\bar\partial( F_{A_2} e^{-\varphi/h} r_2)\rangle\\& =&  \int_{M_0}  \langle \bar\partial^*((|F_{A_1}|^2 - |F_{A_2}|^2)  |F_{A_2}|^{-2}e^{\bar\Phi/h} b),( F_{A_2} e^{-\varphi/h} r_2)\rangle.\end{eqnarray*}
This can now be estimated using the bound for the remainder $r_2$ stated in Proposition \ref{CGO type I} and is of order $o(h)$.

For the term \[ \int_{M_0} (|F_{A_1}|^2 - |F_{A_2}|^2)  \langle|F_{A_1}|^{-2}\bar\partial(F_{A_1} e^{\phi/h} r_1),|F_{A_2}|^{-2}\bar\partial( F_{A_2} e^{-\phi/h} r_2)\rangle\] we can again integrate-by-parts to move all the derivatives to terms involving $r_2$:
{\small \begin{eqnarray}\label{r1r2 term} &&\int_{M_0} (|F_{A_1}|^2 - |F_{A_2}|^2)  \langle\bar\partial(F_{A_1} e^{\phi/h} r_1),|F_{A_1}F_{A_2}|^{-2}\bar\partial( F_{A_2} e^{-\phi/h} r_2)\rangle \\\nonumber&=&  \int_{M_0}  \langle F_{A_1} e^{\phi/h} r_1,\bar\partial^*((|F_{A_1}|^2 - |F_{A_2}|^2)  |F_{A_1}F_{A_2}|^{-2}\bar\partial( F_{A_2} e^{-\phi/h} r_2))\rangle.\end{eqnarray}}
Observe that by Proposition \ref{CGO type I} and \eqref{(L-lambda)u0} we have {\Small\begin{eqnarray*}\|e^{-\phi/h} \Delta e^{\phi/h} r_2\|_{e^{\gamma_0/x}L^2} \leq C(\|d  r_2\|_{e^{{\gamma_0}/{x}}L^2} &+& h^{-1}\|r_2\|_{e^{\gamma_0/x}L^2} \\&+& \|e^{-\phi/h}(L_{X_2,V_2}-\lambda^2) u_{0,2}\|_{e^{\gamma_0/x}L^2})\\&\leq& o(1)\end{eqnarray*}}and we see therefore that \eqref{r1r2 term} is $o(h)$. We can conclude then that \eqref{recovering FA terms} indeed becomes
\begin{eqnarray}
\label{principal term in recovering FA}
\int_{M_0}(|F_{A_1}|^2 - |F_{A_2}|^2)   e^{2i\psi/h} |b|^2 = o(h).
\end{eqnarray}
We are now in a position to prove 
\begin{lemma}
\label{FA has same size}
If $S_{X_1,V_1}(\la) = S_{X_2,V,2}(\la)$
for a fixed $\lambda \in \R\backslash\{0\}$ then for $F_{A_1}$ and $F_{A_2}$
chosen as in \eqref{eq_def_FA} one has $|F_{A_1}|  = |F_{A_2}|$.
\end{lemma}
\begin{proof}
Let $p_0\in M_0$ the critical point of a Morse meromorphic function $\Phi = \varphi + i\psi$ on $M_0\cup\{e_1,\dots,e_N\}$ whose (non-removable) poles are all simple and form precisely the set $\{e_1,\dots,e_N\}$. From Proposition \ref{criticalpoints} we know that such points form a dense subset of $M_0$. If $\{p_0,\dots, p_n\}$ are the critical points of $\Phi$, choose antiholomorphic 1-form $b$ satisfying condition \eqref{condition for b} and apply stationary phase expansion to \eqref{principal term in recovering FA} we see that $|F_{A_1}(p_0)| = |F_{A_2}(p_0)|$. Since $p_0$ can be chosen over an dense subset of $M_0$, the continuity of $F_{A_j}$ completes the proof.\end{proof}
This Lemma leads immediate to the\\
\noindent{\bf Proof of Proposition \ref{same gauge}} By Lemma \ref{FA has same size} we can define the unitary function $\Theta := \frac{F_{A_1}}{F_{A_2}} = \frac{\bar F_{A_2}}{\bar F_{A_1}}$. Note that due to Lemma \ref{FA2 expression}, $\Theta \in 1+ e^{-\gamma/x} W^{1,\infty}(M_0)$ for all $\gamma>0$. We see that $\bar\partial \Theta = i\pi_{0,1}(X_1 - X_2)/\Theta$ while $\partial \Theta = i\pi_{1,0}(X_1 - X_2)/\Theta$. Adding the two identities together we obtain Proposition \ref{same gauge}.\qed

\begin{section}{Determining the Zeroth order term} \label{sec_zero}
In this section we complete the proof of Theorem \ref{thm1} by proving 
\begin{proposition}
\label{recovering zeroth order term}
If $S_{X_1,V_1}(\la) = S_{X_2,V,2}(\la)$ for a fixed $\lambda \in \R\backslash\{0\}$ then $V_1 = V_2$.
\end{proposition}
This will be accomplished with CGO of type II given by \eqref{type II CGO}. To this end, let $u_1$ and $u_2$ be solutions of the from \eqref{type II CGO} with phase $\Phi$ and $-\Phi$ respectively: 
\begin{eqnarray*}
\label{type II CGO}
u_1 = e^{\Phi/h}(a + r_1) + e^{\varphi/h}r_2,\ \ \ u_2 = e^{-\Phi/h}(a + s_1) + e^{-\varphi/h}s_2
\end{eqnarray*}
with $r_2$ and $s_2$ satisfying the estimates
\[\|e^{-\gamma_0/x} r_2\| + h\|e^{-\gamma_0/x} dr_2\|+\|e^{-\gamma_0/x} s_2\| + h\|e^{-\gamma_0/x} ds_2\| \leq C h^{1+ \frac{1}{2}} |\log h|\]
for some $\gamma_0>0$. 

Note that since we have already shown in Proposition \ref{same gauge} that $X_1$ and $X_2$ are gauge equivalent, we may assume without loss of generality that they are actually identical. Therefore, for the CGO $u_1$ and $u_2$, the identity in Lemma \ref{intid} holds with $F_{A_1} = F_{A_2} = F_A = e^{i\alpha}$ to become
\begin{align*}
   0= \int_{M_0}  \big |F_A|^2\langle (Q_1-Q_2 ) \tilde u_1 , \tilde u_2 \big\rangle  =  \int_{M_0}  \big |F_A|^4\langle (Q_1-Q_2 )  u_1 ,  u_2 \big\rangle
\end{align*}
where $\tilde u_j = F_{A} u_j$ and $Q_j = *dX_j + V_j$. We now plug in the expression for $u_1$ and $u_2$ into this identity. Using Lemma \ref{fewestimates} and elementary estimates we obtain
\[ o(h) = \int_{M_0} e^{2i\psi/h} (Q_1 - Q_2) |a|^2\]
Repeating the same argument as in proof of Lemma \ref{FA has same size} we have that $Q_1 = Q_2$ on $M_0$ which implies that $V_1= V_2$ and Proposition \ref{recovering zeroth order term} is verified.

\end{section}

\end{document}